\documentclass[12pt]{amsart}
\usepackage{latexsym}
\usepackage{amsmath} 
\usepackage{amsfonts} 
\usepackage{amssymb} 
\usepackage{amsbsy} 
\usepackage{amsthm}
\usepackage{color}
\usepackage[mathscr]{eucal} 
\usepackage{a4wide}
\def\qed{\relax\ifmmode\hskip2em \Box\else\nobreak\hskip1em $\Box$\fi}

\theoremstyle{plain}

{\end{list}}

\theoremstyle{definition}
\newtheorem{teo}{Theorem}[section]
\newtheorem{pro}[teo]{Proposition}
\newtheorem{lem}[teo]{Lemma}
\newtheorem{cor}[teo]{Corollary}
\newtheorem{den}[teo]{Definition}
\newtheorem{obs}[teo]{Remark}
\newtheorem{prp}[teo]{Propiedad}
\newtheorem{eje}[teo]{Example}
\newtheorem{prps}[teo]{Propiedades}
\newtheorem{obss}[teo]{Remarks}
\newtheorem{ejes}[teo]{Examples}
\newenvironment{teor}{\smallskip\begin{teo}}{\end{teo}\smallskip}
\newenvironment{prop}{\smallskip\begin{pro}}{\end{pro}\smallskip}
\newenvironment{lema}{\smallskip\begin{lem}}{\end{lem}\smallskip}

\newenvironment{defi}{\smallskip\begin{den}}{\end{den}\smallskip}
\newenvironment{obse}{\smallskip\begin{obs}}{\end{obs}\smallskip}

\newenvironment{ejem}{\smallskip\begin{eje}}{\end{eje}\smallskip}

\newenvironment{obses}{\smallskip\begin{obss}}{\end{obss}\smallskip}

\def\a{\alpha}
\def\b{\beta}

\def\ep{\varepsilon}
\def\eps{\varepsilon}

\def\balpha{\boldsymbol{\alpha}}
\def\bbeta{\boldsymbol{\beta}}
\def\bdelta{\boldsymbol{\delta}}
\def\boeta{\boldsymbol{\eta}}
\def\btheta{\boldsymbol{\theta}}

\def\blambda{\boldsymbol{\lambda}}
\def\bomega{\boldsymbol{\omega}}

\def\brho{\boldsymbol{\rho}}
\def\bA{\boldsymbol{A}}
\def\ba{\boldsymbol{a}}

\def\be{\boldsymbol{e}}

\def\bH{\boldsymbol{H}}

\def\bN{\boldsymbol{N}}

\def\bt{\boldsymbol{t}}

\def\bz{\boldsymbol{z}}
\def\bzeta{\boldsymbol{z}}

\def\bw{\boldsymbol{w}}

\def\bB{\boldsymbol{B}}
\def\bR{\boldsymbol{R}}

\def\bk{\boldsymbol{k}}

\newcommand{\FF}{{\mathcal F}}
\newcommand{\OO}{{\mathcal O}}
\newcommand{\CC}{{\mathcal C}}

\def\K{\hbox{\raise2pt\hbox{$\chi$}}}   

\newcommand{\norm}[1]{\left | #1 \right |}
\newcommand{\norma}[1]{\left |\left| #1 \right|\right |}

\DeclareMathOperator{\real}{Re}
\renewcommand{\Re}{\real}

\def\n{\mathbb{N}}      
\def\N{\mathbb{N}}

\def\R{\mathbb{R}}
\def\C{\mathbb{C}}      
\def\c{\mathbb{C}}
\def\adh#1 {\overline{\hbox{$#1$}}}        
\def\inter#1 {\buildrel {\,\circ} \over {#1}}  
\def\fra #1 #2
{\raise.3ex\hbox{$#1$}\kern-.2ex/\kern-.2ex\lower.5ex\hbox{$#2$}}
\def\frs #1 #2
{\raise.3ex\hbox{$\sues #1$}\kern-.3ex\hbox{$\sues /$}
\kern-.3ex\lower.4ex\hbox{$\sues #2$}}
\def\frb #1 #2 {{\des\mathstrut {#1}\over\des\mathstrut {#2}}}
\def\raiz #1 #2 {\root {\raise.3em\hbox{$\sues #1$}} \of {#2}}
\def\listados#1 #2 {{#1}_1,{#1}_2,\ldots,{#1}_{#2}}  
\def\vec#1 #2 {({#1}_1,{#1}_2,\ldots,{#1}_{#2})}  
\def\s#1 #2 {\{{#1}_{#2}\}_{{#2}=1}^\infty}         
\def\suc #1 #2 #3 {\{#1\}_{#2=#3}^\infty}   
\def\F#1 #2 #3 {\boldsymbol{#1}(x,y)=\left({#2},{#3}\right)}  
\def\Fe#1 #2 #3 #4 {\boldsymbol{#1}(x,y,z)=\left({#2},{#3},{#4}\right)}  

\def\nor#1 {\left\Vert {#1}\right\Vert}   

\def\ds{\displaystyle}
\def\des{\displaystyle}
\def\sst{\scriptstyle}
\def\sues{\scriptstyle}

\def\linea{\mbox{}\par\nopagebreak}

\newcommand{\sase}{strong asymptotic expansion}
\newcommand{\ase}{asymptotic expansion}
\newcommand{\ases}{asymptotic expansions}
\newcommand{\sases}{strong asymptotic expansions}
\newcommand{\PL}{Phragm\'{e}n-Lindel\"{o}f}

\makeatletter          
\def\@@and{ and }
\renewcommand{\andify}{%
\nxandlist{\unskip, }{\unskip{} \@@and~}{\unskip \@@and~}}
\def\and{\unskip{ }\@@and{ }\ignorespaces}
\makeatother           

\begin{document}

\title{Strong asymptotic expansions in a multidirection}
\author{Alberto Lastra}
\address[Alberto Lastra]{Dpto. An\'{a}lisis Matem\'{a}tico y Did\'{a}ctica de la Matem\'{a}tica \\
Facultad de Ciencias \\
Universidad de Valladolid\\
Prado de la Magdalena, s/n\\
47005 Valladolid\\
Spain}
\email{alastra@am.uva.es}

\author{Jorge Mozo-Fern\'{a}ndez}
\address[Jorge Mozo Fern\'{a}ndez]{Dpto. Matem\'{a}tica Aplicada \\
Facultad de Ciencias \\
Universidad de Valladolid\\
Prado de la Magdalena, s/n\\
47005 Valladolid\\
Spain}
\email{jmozo@maf.uva.es}
\thanks{First and third author partially supported by Ministerio de Ciencia e Innovaci\'{o}n (Spain) under Project MTM2009-12561\\
Second author partially supported by Ministerio de Ciencia e Innovaci\'{o}n (Spain) under Project MTM2010-15471.\\
First author partially supported by Ministerio de Educaci\'{o}n (Spain): Programa
Nacional de Movilidad de Recursos Humanos, Plan Nacional de I-D+i 2008-2011.
}

\author{Javier Sanz}
\address[Javier Sanz]{Dpto. An\'{a}lisis Matem\'{a}tico y Did\'{a}ctica de la Matem\'{a}tica \\
Facultad de Ciencias \\
Universidad de Valladolid\\
Prado de la Magdalena, s/n\\
47005 Valladolid\\
Spain}
\email{jsanzg@am.uva.es}
\date{\today}
\begin{abstract}

In this paper we prove that, for asymptotically bounded holomorphic
functions defined in a polysector in ${\mathbb C}^n$, the existence of
a strong asymptotic expansion in Majima's sense following a single
multidirection towards the vertex entails (global) asymptotic expansion
in the whole polysector. Moreover, we specialize this result for
Gevrey strong asymptotic expansions. This is a generalization of a
result proved by A. Fruchard and C. Zhang for asymptotic expansions in
one variable, but the proof, mainly in the Gevrey case, involves
different techniques of a functional-analytic nature.

\vspace{0.5cm}

keywords: Strong asymptotic expansions; Gevrey asymptotics;
Laplace transform; Banach spaces analytic functions 

\vspace{0.5cm}


\end{abstract}





\maketitle
%
%

\section{Introduction}

In 1886 H. Poincar\'e put forward the concept of asymptotic expansion for holomorphic complex functions $f$ defined in an open sector $S\subseteq\C$ with vertex at 0, by associating to $f$ a formal power series $\hat{f}=\sum_{n\ge 0}a_nz^n\in\C[[z]]$ whose partial sums $\hat{f}_{N}=\sum_{n\le N-1}a_nz^n$ suitably approximate $f$ on every subsector $T$ of $S$. The most important instance consists of the so-called Gevrey asymptotic expansions of a positive order $k$, in which case $f-\hat{f}_{N}$ is bounded in $T$ by $CA^{N}N!^k|z|^{N}$ for positive constants $C,A$ depending only on $T$. The infimum of the constants $A>0$ allowed in the previous expression is known as the Gevrey type of the asymptotic expansion. Gevrey asymptotics incessantly appear in the theory of algebraic ordinary differential equations and of meromorphic systems of differential equations at an irregular singular point. In this context, summability theory has received increasing attention from many authors such as J.-P. Ramis, W. Balser, B.L.J. Braaksma, B. Malgrange, Y. Sibuya and others, since the seventies (\cite{Balser} is a standard reference).

In 1999, A. Fruchard and C. Zhang (Theorem 1 in \cite{Fruzha}) proved that, for a holomorphic function, defined in a sector $S$ and bounded in every subsector of $S$, it is sufficient for the asymptotic expansion to exist that the approximation occurs just in  a direction in $S$, i.e. for the elements in $S$ having a fixed argument. So, asymptotic expansion following a direction entails (classical, global) asymptotic expansion. Moreover, if the asymptotic expansion following direction $\theta$ is of $k$-Gevrey order and fixed type, the  asymptotic expansion remains $k$-Gevrey in $S$, and the type following any other direction may also be determined (for more details we refer to~\cite{Fruzha}).
The proofs of these two results rest, on one hand, on \PL-like results and, on the other hand, on the Borel-Ritt and Borel-Ritt-Gevrey theorems, respectively. Our main aim in this work is to give a generalization of this result to the several variables setting (Theorem~\ref{teorfinal}).

R. G\'erard and Y. Sibuya~\cite{gersib} extended in 1979 the concept of asymptotic expansion to the case of holomorphic functions of several variables defined in polysectors (products of sectors), but their definition turned out not to be stable under derivation, unlike Poincar\'e's one (see~\cite{hersanz} for a counterexample due to J. A. Hern\'andez, F. L\'opez and S. P\'erez-Cacho). In 1983 H. Majima~\cite{Majima,Majima2} defined the so-called strong asymptotic expansions, a technical and difficult-to-handle notion for which, however, the following equivalence was found~\cite{jesusisla}: a function $f$ admits strong asymptotic expansion in a polysector $S$ if, and only if, the function and all its derivatives are asymptotically bounded, i.e., bounded in each bounded subpolysector of $S$. Hence, this concept overcomes the lack of G\'erard-Sibuya's definition, and it is considered to be the good generalization to several variables of the classical concept. The previous equivalence remains valid in the Gevrey case, introducing appropriate bounds for the derivatives, see~\cite{Haraoka}. The asymptotic information for a function $f$ is carried in this case by a coherent, so-called total family of functions, denoted by $\mathrm{TA}(f)$ (see Definitions~\ref{defsase} and \ref{deficohe}), whose elements can be obtained, as it happens in the one variable case with the $a_n$ above, as limits of the derivatives of $f$ with respect to some of the variables when they tend to zero (see~(\ref{familiatotal})). So, interpolation results such as Borel-Ritt and Borel-Ritt-Gevrey theorems in this context should start from families such as $\mathrm{TA}(f)$ instead of just series $\sum_{\bN\in\N_0^n}f_{\bN}\bz^{\bN}$ in $n>1$ variables (whose coefficients are just a part of the total family, see Remark~\ref{obseFA}). This makes a big difference between the situations in one or several variables, and causes many of the intricacies coming next, as we will try to explain henceforth.

The paper is organized as follows:

Section 2 is devoted to fix some notation that will be used throughout the paper. In
Section 3 we establish some preliminaries on general and Gevrey strong asymptotic expansions, both in a global sense and following a product of directions (a \emph{multidirection}, in this paper). In a multidirection, null strong asymptotic expansion and exponential decreasing are proved to be equivalent, see Proposition~\ref{propoexpoplana}. Also, the main result of Fruchard and Zhang is presented (Theorem~\ref{FZ1var}).
Section 4 is devoted to recalling or stating several preparatory lemmas, specially a \PL\ theorem in several variables.

Section 5 contains the main results of the paper. From Lemma~\ref{cotaeninfinito} to Remark~\ref{59} we prove that null strong asymptotic expansion in a multidirection amounts to global null asymptotic expansion, both in the general and in the Gevrey cases. An important part of the paper is based in Lemma \ref{asquevv}, that is not an obvious generalization of the corresponding result in one variable, Lemma 2 in \cite{Fruzha}. The need to control the size of the constants involved, in order to make induction on the number of variables, forced us to develop a completely different proof, of a prominently technical nature, for a refined and subtle version of Lemma 2 in \cite{Fruzha}, appearing here as Lemma~\ref{asqueroso}. Regarding the general multidimensional version of our main result, Theorem~\ref{teorfinal} part (i), one can draw an exact analogy with Theorem 1 in~\cite{Fruzha}, since all is needed is a Borel-Ritt theorem for strong asymptotics which interpolates total families, like the ones in~\cite{jesusisla,javifelix}.

However, in the Gevrey case things become more complicated, as one needs to take care of the types.
In the one variable case, for a function $f$ with 1-Gevrey \ase\ $\hat{f}$ in a direction with type $R>0$ one can readily associate, by means of the truncated Laplace transform of the 1-Borel transform of $\hat{f}$, another function with the same \ase\ in a sector with opening $\pi$ and whose type may be explicitly determined in every direction in the sector. This changes drastically in the several variables situation, since the same procedure will provide a function with the same formal series of \sase, but with a possibly different total family. This is explained from Definition~\ref{defiserie1GtipoR} to Theorem~\ref{BRG}, where, following an idea of Haraoka~\cite{Haraoka}, we use the truncated Laplace transform to solve a multidimensional Borel-Ritt-Gevrey interpolation problem for series of \sase, while making precise estimations on the type in every multidirection. This method works equally well for series with coefficients in a complex Banach space, but fails if they are allowed to belong to a Fr\'echet space, since the very concept of Gevrey series will be no longer clear in this general case. So, our next objective will be to settle the problem of interpolating total families in a Banach space framework.

Precise interpolation results (i.e.\ with estimations for the type in every multidirection) for strong asymptotic expansions are not available neither in the works of H. Majima, Y. Haraoka and others, nor in the papers on closely related subjects that exist in the literature, especially dealing with extension results for ultradifferentiable or ultraholomorphic classes (see, for example, \cite{chaucho,LastraSanz,petzsche,javi2,schmets,Thilliez1,thilliez}). In order to remedy this lack, we note first that, following \cite{Haraoka,jesusisla}, the derivatives of functions with $\boldsymbol{1}$-Gevrey \sase\ are suitably bounded, and this fact permits the introduction of a Banach structure in the space of such functions, see Definition~\ref{defW1RSE}, where the type is allowed to depend on the multidirection (a similar idea was already applied in  \cite{javi1,javi2}, but there the type was assumed to be constant). Secondly, we replace $\textrm{TA}(f)$
by its first order subfamily, $\textrm{TA}_1(f)$, as indicated in Definition~\ref{deffamiliaprimerorden}.
Before obtaining a Borel-Ritt-Gevrey result in this context (Lemma~\ref{521}), we need to make a study of the growth of the derivatives of the function built in Theorem~\ref{BRG} in order to properly identify the Banach space to which the function belongs (see Proposition~\ref{transLaplaW1Rtilde}). As we see in this result, there is a price to pay, concerning the type of the Gevrey asymptotic expansion: The precise knowledge of the exact type of the asymptotic expansion does not allow one to know exact bounds for the derivatives. This does not represent a serious inconvenient in most of the applications, where the precise type is of moderate interest, but will cause technical difficulties in the present paper. Mainly, there is a loss of precision that in particular forced us to control the type by a term of order $\cos (\theta -\theta_0)^2$, instead of $\cos (\theta -\theta_0)$, as would be desirable. The existence of better interpolation results would make this proof easier, and the bounds involved, less complicated.
The second part of Theorem~\ref{teorfinal} is the $\boldsymbol{1}$-Gevrey multidimensional version of Fruchard-Zhang's result. We depart from a function $f$ admitting $\boldsymbol{1}$-Gevrey \sase\ following a multidirection $\btheta_0$ with a type $\bR_0(\btheta_0)$, and with an associated first order family whose elements belong to suitable Banach spaces with types $R_j$, $j=1,\ldots,n$. Then, we prove that $f$ admits $\boldsymbol{1}$-Gevrey \sase\ in the whole polysector with a type in every multidirection $\btheta$ which depends on $\bR_0(\btheta_0)$, the $R_j$ and the function $g$ in Proposition~\ref{transLaplaW1Rtilde}, which was related to our loss of precision.

%

Finally, we point out that related results could possibly be obtained for other types of asymptotic expansions available in the literature, for instance, in other ultraholomorphic classes of functions. Nevertheless, the absence of appropriate integral expressions for the corresponding interpolating operators in this context makes very difficult and technical its treatment, as can be seen in the works of J. Chaumat, A.-M. Chollet~\cite{chaucho} and V. Thilliez~\cite{thilliez}. Monomial asymptotic expansions developed by M. Canalis-Durand, J. Mozo and R. Sch\"{a}fke  in \cite{CDMS}, in order to treat singularly perturbed problems, are another class of asymptotic expansions where similar results could undoubtedly be obtained.

\section{Notations}

We set $\N:=\{1,2,\ldots\}$ and $\N_{0}:=\N\cup\{0\}$. Let $\balpha=(\alpha_1,\alpha_2,\ldots,\alpha_n)$,
$\bbeta=(\beta_1,\beta_2,\ldots,\beta_n)\in\N_0^{n}$ be two multiindices,
$m\in[0,\infty)$, $\bt=(t_1,t_2,\ldots,t_n)\in\R^n$ and $\bzeta=(z_1,z_2,\ldots,z_n)\in\c^n$. We set
\begin{align*}
&\balpha+\bbeta=(\a_{1}+\b_{1},\a_{2}+\b_{2},\ldots,\a_{n}+\b_{n}),
&&|\balpha|=\a_{1}+\a_{2}+\ldots+\a_{n},\\
&\balpha\le\bbeta\Leftrightarrow\a_{j}\le\b_{j}\hbox{ for every }j,
&&\balpha<\bbeta\Leftrightarrow\a_{j}<\b_{j}\hbox{ for every }j,\\
&\boldsymbol{1}=(1,1,\ldots,1),
&&\boldsymbol{z}^{\balpha}=z_{1}^{\a_{1}}z_{2}^{\a_{2}}\cdots z_{n}^{\a_{n}},\\
&|\boldsymbol{z}|^{\boldsymbol{t}}=
|z_{1}|^{t_{1}}|z_{2}|^{t_{2}}\ldots|z_{n}|^{t_{n}},
&&D^{\balpha}=\frac{\partial^{\balpha}}{\partial\boldsymbol{z}^{\balpha}}=
\frac{\partial^{|\balpha|}}{\partial z_{1}^{\a_{1}}\partial z_{2}^{\a_{2}}\ldots\partial z_{n}^{\a_{n}}},\\
&m\boldsymbol{t}=(mt_{1},mt_{2},\ldots,mt_{n}),
&&m^{\balpha}=m^{|\balpha|},\\
&\balpha!=\a_{1}!\a_{2}!\cdots\a_{n}!,
&&\be_j=(0,\ldots,\stackrel{j)}{1},\ldots,0),\\
&\arg(\bzeta)=(\arg(z_1),\ldots,\arg(z_n)),
&&\cos(\bt)=(\cos(t_1),\ldots,\cos(t_n)).
\end{align*}
A sector in $\C$ will be a set
$$ S=S(\a,\b;\rho)=\{z\in\C:0<|z|<\rho,\hbox{arg}(z)\in(\a,\b)\},$$
where $\a,\b\in\R$, $\a<\b$, and $\rho\in(0,\infty]$.
\par
We say a sector $T$ is a (bounded and proper) subsector of $S$, and write $T\prec S$, whenever $T$ is bounded and $\overline{T}\setminus\{0\}\subseteq S$.\par
A polysector $S=\prod_{j=1}^{n}S_j\subseteq \C^n$ is a cartesian product of sectors.
We say $T=\prod_{j=1}^{n}T_{j}$ is a (bounded and proper) subpolysector of $S$, and write $T\prec S$, if $T_{j}\prec S_j$ for $j=1,\ldots,n$.\par
Given $\balpha=(\a_{1},\ldots,\a_n),\,\bbeta=(\b_1,\ldots,\b_n)\in\R^{n}$, with $\a_j<\b_j$ for every $j$, and $\brho=(\rho_1,\ldots,\rho_n)\in(0,\infty]^n$, $S(\balpha,\bbeta;\brho)$ denotes the polysector $\prod_{j=1}^{n}S(\a_j,\b_j;\rho_j)$.
If $\btheta=(\theta_1,\ldots,\theta_n)$ is such that $\a_j<\theta_j<\b_j$ for every $j$, we say $\btheta$ is a multidirection in $S(\balpha,\bbeta;\brho)$.\par
$\partial S$ denotes the boundary of $S$, while $\partial_d S$ stands for the distinguished boundary, i.e.,
$$\partial_d S=\{(z_1,\ldots,z_n)\in \overline{S}: z_j\in\partial S_j\hbox{ for every }j\in\{1,\ldots,n\}\}.$$
\indent $\OO(S)$, resp.\ $\CC(S)$, will stand for the space of holomorphic, resp.\ continuous, complex functions defined in $S$.\par
For a set $B\subset\C\setminus\{0\}$, we put $B^{-1}:=\{z\in\C:z^{-1}\in B\}$.\par

For $n\in\n$, we put $\mathcal{N}=\{1,2,\ldots,n\}$.
For $J\subset\mathcal{N}$, $\#J$ denotes its cardinal number, and
$J':=\mathcal{N}\setminus J$. For $j\in\mathcal{N}$ we use ${j\,}'$ instead of $\{j\}'$.
In particular, we shall use these conventions for multiindices.\par
Given $\bzeta\in\C^n$ and a nonempty subset $J$ of $\mathcal{N}$, we write
$\bzeta_J$ for the restriction of $\bzeta$ to $J$, considering $\bzeta$
as an element of $\C^{\mathcal{N}}$. Similarly, if $S=\prod_{j=1}^nS_j$
is a polysector of $\c^n$, then $S_J:=\prod_{j\in J}S_j\subset\c^J$.\par
Finally, $D(\boldsymbol{z},\boldsymbol{R})$ will denote the polydisc with centre $\bzeta\in\C^n$ and polyradius $\boldsymbol{R}\in(0,\infty)^n$.

\section{Preliminaries}
Strong \ases\ were introduced by H. Majima \cite{Majima,Majima2}. This notion generalizes to several complex variables the concept of \ase\ defined by H.\ Poincar\'{e} in 1886. Let us recall in this Section the main definitions and state the basic result of
the paper.

\begin{defi}\label{defsase}
Let $n\in \N$, $n\geq 1$, and $S\subseteq \C^n$ be a polysector. A function $f\in {\mathcal O}(S)$ has  a \sase\ in $S$ if there exists a family
\begin{equation}\label{e349}
\mathcal{F}=\{f_{\bN_J}\in\OO(S_{J'}): \emptyset\neq J \subset \mathcal{N}, \bN_J\in\N_{0}^J\},
\end{equation}
($f_{\bN_J}\in  \C$ if $J=\mathcal{N}$), such that, if we define for every $\bN\in \N_0^n$ the function
$$
\textrm{App}_{\bN}(\mathcal{F})(\bzeta):=\!\sum_{\emptyset\neq J\subset \mathcal{N}}\!(-1)^{\#J+1}\!\sum_{\scriptstyle\bH_J\in\N_{0}^J\atop\scriptstyle
\bH_J<\bN_J}f_{\bH_J}(\bzeta_{J'})\bzeta_J^{\bH_J},\qquad \bzeta\in S,
$$
then, for every subpolysector $T\prec S$ and every $\bN\in \N_0^n$, there exists $c=c(\bN,T)>0$ with
\begin{equation}
\label{e359}
|f(\bz)-\textrm{App}_{\bN}(\mathcal{F})(\bz)|\leq c|\bz|^{\bN}, \qquad \bz\in T.
\end{equation}
Following Haraoka \cite{Haraoka}, $f$ is said to have ${\boldsymbol{1}}$-Gevrey \sase\ in $S$ if the constant $c=c(\bN,T)$ in~(\ref{e359}) may be chosen in the form
$$
c(\bN,T)=c\bA^{\bN}\bN!,
$$
for certain $c=c(T)>0$ and $\bA=\bA(T)\in(0,\infty)^n$ depending only on $T$.
\end{defi}

Let us note that ${\mathcal F}$ is uniquely determined by $f$. In fact, if $T\prec S$, $\emptyset \neq J\subseteq {\mathcal N}$ and $\bN_J\in\N_0^J$,
\begin{equation}\label{familiatotal}
\lim_{\scriptstyle \bz_{J}\to \textbf{0}_{J}\atop \scriptstyle \bz_{J}\in T_{J}}\frac{D^{(\bN_J,\mathbf{0}_{J'})}f(\bz)}{\bN_{J}!}=f_{\bN_{J}}(\bz_{J'}),
\end{equation}
the limit being uniform in $T_{J'}\prec S_{J'}$ whenever $J\neq\mathcal{N}$. We call $\mathcal{F}$ the total family of \sase\ of $f$, and we denote it from now on by $\textrm{TA}(f)$.

We will say that $f$ has null \sase\ if all the elements in $\textrm{TA}(f)$ identically vanish.

$\textrm{TA}(f)$ turns out to be a coherent family, in the following sense:

\begin{defi}\label{deficohe}
A family $\FF$ as in (\ref{e349}) is coherent in $S$ if for every $T\prec S$, every disjoint nonempty $J,L\subseteq {\mathcal N}$, and every $\bN_J\in \N_0^J$, $\bN_L\in \N_0^L$, we have
$$\lim_{\scriptstyle\bz_{L}\to \mathbf{0}_{L}\atop\scriptstyle
\bz_{L}\in T_{L}}\frac{D^{(\bN_{L},\mathbf{0}_{(J\cup L)'})}f_{\bN_{J}}(z_{J'})}{\bN_{L}!}=f_{(\bN_{J},\bN_{L})}(z_{(J\cup L)'}),$$
the limit being uniform in $T_{(J\cup L)'}$ (if $J\cup L\neq {\mathcal N}$).
\end{defi}

For $n=1$, $\textrm{TA}(f)$ reduces to a sequence of complex numbers $\{a_n\}_{n\in\N_0}$, the coefficients of the formal power series $\hat{f}=\sum_{n\in\N_0}a_nz^n$ of asymptotic expansion in one variable. In this situation we will write $f\sim\hat{f}$, or $f\sim_k\hat{f}$ in case the asymptotic expansion is of Gevrey order $k>0$. Of course, coherence is always satisfied when $n=1$. 

Next we define, generalizing \cite{Fruzha}, a notion of \ase\ following a multidirection.

\begin{defi} \label{def376}
Let $n\in \N$, $S\subseteq \C^n$ a polysector, $f\in \OO (S)$, and $\btheta=(\theta_1,\ldots ,\theta_n )$ a multidirection in $S$. $f$ is said to have a \sase\ following $\btheta$ if there exists a coherent family $\FF$ on $S$ such that, for every $\bN\in\N_0^n$, there exists $c=c(\bN)>0$ verifying (\ref{e359}) for every $\bz\in S$ with $\arg(\bz)= \btheta$.
\end{defi}

Let us state now some definitions and a result relating null ${\boldsymbol{1}}$-Gevrey strong asymptotics to exponential decreasing, both following a multidirection.

\begin{defi}
Let $S\subseteq \C^n$ be a polysector and $\btheta$ a multidirection in $S$. We say $f$ has a ${\boldsymbol{1}}$-Gevrey \sase\ of type $\boldsymbol{R}= (R_1,\ldots,R_n)\in (0,\infty )^n$ following $\btheta$ if $f$ has a \sase\ following $\btheta$ and, moreover, the
constant $c$ that appears in Definition \ref{def376} is given as follows: for every $\delta>0$, there exists $C_1=C_1( \delta )>0$, such that
$$
c=c(\bN)=C_1\Big(\frac{1}{R_{1}}+\delta\Big)^{N_{1}}\cdot\ldots\cdot\Big( \frac{1}{R_{n}}+\delta\Big)^{N_{n}}\cdot \bN !,$$
for every $\bN=(N_1,\ldots,N_n)\in \N_0^n$.
\end{defi}

\begin{obse}
If $f$ and $g$ have ${\boldsymbol{1}}$-Gevrey \sase\ of types $\boldsymbol{R}= (R_1,\ldots,R_n)$ and $\widetilde{\bR}= (\widetilde{R}_1,\ldots,\widetilde{R}_n)$, respectively, following $\btheta$, then $f+g$ has ${\boldsymbol{1}}$-Gevrey \sase\ of type $\widehat{R}= (\widehat{R}_1,\ldots,\widehat{R}_n)$ following $\btheta$,
where $\widehat{R}_j=\min(R_j,\widetilde{R}_j)$, $j\in\mathcal{N}$.
\end{obse}

\begin{defi} \label{defexponplana}
Let $S\subseteq \C^n$ be a polysector, $f\in \OO (S)$, and $\btheta$ a multidirection in $S$. $f$ is said to be exponentially flat, or exponentially decreasing, of type $\boldsymbol{R}=(R_1,\ldots,R_n)\in (0,\infty )^n$ following $\btheta$ if for every $\delta>0$, there exists $M>0$ with
\begin{equation}\label{cotaexponencialplana}
|f(\bz)|\le Me^{-\frac{R_1-\delta}{|z_1|}-\ldots-\frac{R_n-\delta}{|z_n|}},
\end{equation}
for every $\bz=(z_1,\ldots,z_n)\in S$ with $\arg (\bz)= \btheta$.
\end{defi}

\begin{obse}\label{tipo0}
In the one variable case, one agrees that an exponentially flat function of type~0 following a given direction is just a function which is bounded in that direction. So, if in the previous Definition some of the components of $\bR$ were null, it would mean that we can delete the corresponding terms in the exponent appearing in the estimates~(\ref{cotaexponencialplana}).
\end{obse}

These definitions extend those in the one-variable case as stated in \cite{Fruzha}. Next result is essentially contained in \cite{javi1}: we state and prove it here because we need precise estimations of the types $\bR$.

\begin{prop}  \label{propoexpoplana}
Let $S\in \C^n$ be a polysector, $f\in \OO (S)$, $\btheta $ a multidirection in $S$, and $\bR\in (0,\infty )^n$. The following conditions are equivalent:
\begin{enumerate}
\item[(i)] $f$ has null \sase\ of Gevrey order $\boldsymbol{1}$ and type $\bR$ following $\btheta$.
\item[(ii)] $f$ is exponentially flat of type $\bR$ following $\btheta$.
\end{enumerate}
\end{prop}

\begin{proof} (i)$\Rightarrow$(ii). Given $\delta_1>0$, there exists $C>0$ such that, for every $\bN=(N_1,\dots,N_n)\in \N_0^n$,

$$|f(\bz)|\le C\prod_{j=1}^{n}\Big(\Big(\frac{1}{R_{j}}+\delta_1\Big)^{N_{j}}N_{j}!|z_{j}|^{N_{j}}\Big)$$
on $\btheta$. Stirling's Formula shows that, if $\eps >0$, a constant $C'>0$ exists with
$$N_j !\le C'(e^{-1}(1+\ep))^{N_j}N_j^{N_j}.
$$
Then,
$$|f(\bz)|\le CC'\prod_{j=1}^{n}\Big(\Big(\frac{1}{R_{j}}+\delta_1\Big)e^{-1}(1+\eps)N_{j}|z_{j}|\Big)^{N_{j}}.$$
Denote $A_j=\left( \frac{1}{R_j}+\delta_1 \right) e^{-1} (1+\eps )$. The real function $g(x)=(Ax)^x$ ($A>0$, $x>0$), takes its minimum value at $x= (Ae)^{-1}$. So, put $x_j=(A_j |z_j |e)^{-1}$, and take $N_{0j}\in\N_0$ with $N_{0j}\leq x_j< N_{0j}+1$. We have
\begin{align*}
A_{j}^{N_{0j}}N_{0j}^{N_{0j}}|z_{j}|^{N_{0j}}&\le \Big(A_{j}x_{j}|z_{j}|\Big)^{N_{0j}}= e^{-N_{0j}}< e e^{-x_{j}}\\
&=e \exp\Big(-\frac{1}{\Big(\frac{1}{R_{j}}+\delta_1\Big)(1+\ep)}\frac{1}{|z_{j}|}\Big).
\end{align*}
So, a constant $C''>0$ exists such that
$$|f(\bz)|\le C''\prod_{j=1}^{n}\exp\Big(-\frac{1}{\Big(\frac{1}{R_{j}}+\delta_1\Big)(1+\ep)}\frac{1}{|z_{j}|}\Big).$$
Now, given $\delta>0$, take positive $\delta_1$, $\eps$ such that $(1+\eps )\left( \frac{1}{R_j}+\delta_1 \right)<\frac{1}{R_j-\delta}$, and conclude.

\medskip
(ii)$\Rightarrow$(i). Conversely, for $\delta>0$, let us take $\delta_1>0$ such that $\delta_1<\min_j R_j$, and
\begin{equation}\label{deltaydelta1}
\frac{1}{R_j-\delta_1}\le \frac{1}{R_j}+\delta,\quad j=1,2,\ldots,n.
\end{equation}
According to Definition \ref{defexponplana}, there exists $M>0$ such that, for every $\bN= (N_1,\ldots,N_n)\in \N_0^n$, and every $\bz\in S$ with $\arg (\bz )=\btheta$, we have
$$\Big|\frac{f(\bz)}{\bz^{\bN}}\Big|\le M\prod_{j=1}^{n}|z_{j}|^{-N_j}\exp\Big(-\frac{R_j-\delta_1}{|z_j|}\Big).$$
If $N\in \N_0$ and $H>0$, the function $g(x)=x^{-N}\cdot \exp (-H/x)$ is bounded above by $\left[ \frac{N}{eH}\right]^N$ on $(0,\infty )$ (where $\left[ \cdot \right]$ denotes integer part). So
$$\Big|\frac{f(\bz)}{\bz^{\bN}}\Big|\le M\prod_{j=1}^{n}\Big[\frac{N_j}{e(R_j-\delta_1)}\Big]^{N_j}\le M_2 \bN!\prod_{j=1}^{n}
\Big(\frac{1}{R_j-\delta_1}\Big)^{N_j},$$
for some $M_2>0$ (again, by Stirling's Formula). Conclude using (\ref{deltaydelta1}).
\end{proof}

The main purpose of this paper will be to generalize the main result of \cite{Fruzha} to several complex variables. We recall here, for completeness, this result. We need first a definition.

\begin{defi}
Let $S$ be a polysector in $\C^n$. A function $f:S\to\C$ is said to be asymptotically bounded in $S$ if it is bounded on every subpolysector $T\prec S$.
\end{defi}

\begin{teor}[Fruchard-Zhang \cite{Fruzha}, Thm.\ 1]\label{FZ1var}
Let $S=S(\alpha, \beta; \rho )$ be a sector, and $f\in \OO (S)$ asymptotically bounded. Let $\hat{f}(z)\in \C [[z]]$ and $\theta$ be a direction in $S$.
\begin{enumerate}
\item[(i)] If $f\sim \hat{f}$ following $\theta$, then $f\sim \hat{f}$ in $S$.
\item[(ii)] The same result holds for 1-Gevrey asymptotics. More precisely, if $f\sim_1 \hat{f}$ following $\theta_0$ with type $R(\theta_0)>0$, then along any direction $\theta $ of $S$, $f\sim_1 \hat{f}$ following $\theta$ with type $R(\theta )$, where $R(\theta)$ is defined as follows:
$$
R(\theta)=\begin{cases}
	 R(\theta_0)\frac{\sin(\theta-\a)}{\sin(\a'-\a)} &\text{ if }\theta\in(\a,\a'] \\
	R(\theta_0)&\text{ if }\theta\in[\a',\b'] \\
	 R(\theta_0)\frac{\sin(\theta-\b)}{\sin(\b'-\b)} &\text{ if }\theta\in[\b',\b),
	\end{cases}
$$
where $\a'= \min \{ \theta_0, \a+\frac{\pi}{2}\}$, $\b'=\max \{ \theta_0, \b -\frac{\pi}{2}\}$.
\end{enumerate}
\end{teor}

\section{Classical results revisited}

The following Lemma is an extension to the boundary of a
well-known result of complex variables. It will allow us to make
some reasonings by induction.

\begin{lema}  \label{lemafelix}
Let $S=S_1\times \cdots \times S_n\subseteq \C^n$ be a
polysector, $f:\overline{S}\rightarrow \C$ continuous and holomorphic
in $S$. If $\ba= (a_2,\ldots ,a_n)\in \partial_d (S_2\times
\cdots \times S_n)$, then $f_{\ba}: S_1\rightarrow \C$ defined by
$f_{\ba} (z)= f(z,\ba)$ is holomorphic in $S_1$ (and continuous in
$\overline{S}_1$).
\end{lema}

\begin{proof}
Only the holomorphic part is interesting now. Take a sequence
$\{ \ba_m \in S_2\times \cdots \times S_n\}_{m}$ converging to
$\ba$, and $K\subseteq S_1$ a compact set. Consider the
holomorphic functions $f_m :S_1\rightarrow \C$ defined by $f_m
(z)= f(z,\ba_m)$.

If $L= \{ \ba_m\}_{m\in \N}\cup \{ \ba\}$, $f\mid_{K\times L} :
K\times L \rightarrow \C$ is uniformly continuous. So, given
$\eps>0$, $\exists\, \delta >0$ such that if $(z,\bw)$, $(z',
\bw')\in K\times L$, $\norma{(z,\bw)- (z',\bw')}<\delta$ implies
$\norm{f(z,\bw) - f(z',\bw')}<\eps$. If $m$ is big enough such
that  $\norma{\ba-\ba_m}<\delta$, then $\norm{f_{\ba}(z)-
f_m(z)}<\eps$ for every $z\in K$, and so, $\{ f_m\}_m$ converges uniformly in the compact subsets of $S_1$ to $f_{\ba}$.
\end{proof}

\begin{obse}
Lemma \ref{lemafelix} remains valid  if some (or all) components
of $\ba$ belong to the interior of the sectors. It allows to
generalize some well-known results from complex analysis, as
follows.
\end{obse}

\begin{lema}
Let $S$
be a polysector and $f\in
\OO (S)\cap {\mathcal C} (\overline{S})$. If $\norm{f(\bz)}\leq M$ for
all $\bz \in \partial_d S$, then $\norm{f(\bz)}\leq M$ in
$\overline{S}$.
\end{lema}

\begin{proof}
We reason by induction in $n$. If $n=1$, there is nothing to prove. If $n>1$,
let $\bz= (z_1,\ldots ,z_n)\in \partial S$. Assume, without loss of generality, that $z_1\in
\partial S_1$ ($S=S_1\times \cdots \times S_n)$, and consider
$f_{z_1} (w_2,\ldots ,w_n)\in \OO (S_2\times \cdots \times
S_n)\cap {\mathcal C}(\overline{S_2\times \cdots \times S_n})$
constructed as in Lemma \ref{lemafelix}. We have $\norm{f_{z_1}
(w_2,\ldots ,w_n)}\leq M$ on $\partial_d (S_2\times \cdots
\times  S_n)$, and by induction hypothesis, $\norm{f_{z_1}
(w_2,\ldots ,w_n)}\leq M$ on $\overline{S_2\times \cdots \times
S_n}$. Then $\norm{f(\bz)}\leq M$ on $\partial S$, and so, on
$S$.
\end{proof}

Finally, we state an appropriate version of Phragm\'en-Lindel\"{o}f
principle in several variables.

\begin{teor}[Phragm\'en-Lindel\"{o}f]\label{phragmenlindelofvv}
Let $S=\prod_{i=1}^n S_i = S(\balpha,\bbeta; \brho )$ be
a polysector, $f\in \OO (S)\cap {\mathcal C} \big(
\prod_{j=1}^n (\overline{S}_j\setminus \{ 0\})\big)$, and
assume $\norm{f(\bz)}\leq M$ for every $\bzeta\in\prod_{j=1}^n\big(\partial S_j\setminus\{0\}\big)$. Assume also that, for
a certain $j\in \{ 1,\ldots ,n\}$, there are strictly positive
functions
$$
K(\bzeta_{{j\,}'})=K(z_1,\ldots,\hat{z}_j,\ldots,z_n), \text{ and }
L(\bzeta_{{j\,}'})=L(z_1,\ldots,\hat{z}_j,\ldots,z_n),
$$
and $a_j<\dfrac{\pi}{\beta_j-\alpha_j}$ such that
\begin{equation}
\label{e504}
|f(\bz)|\le K(z_1,\ldots,\hat{z}_j,\ldots,z_n) \exp{\Big(\frac{L(z_1,\ldots,
\hat{z}_j,\ldots,z_n)}{|z_{j}|^{a_{j}}}\Big)}
\end{equation}
on $S$. Then $\norm{f(\bz)}\leq M$ on $S$.
\end{teor}

\begin{proof}
Consider $\bz_0= (z_{01}, \ldots ,z_{0n})\in S$, and $f_{\bz_0,j}
\in \OO (S_j)\cap {\mathcal C}(\overline{S_j}\setminus \{ 0\})$
defined by
$$
f_{\bz_0,j} (z_j)= f(z_{01}, \ldots ,z_j,\ldots, z_{0n}).
$$
Applying Phragm\'en-Lindel\"{o}f principle in one variable (more precisely, the
version stated in \cite[\S 2]{Fruzha}), we have that $\norm{f_{\bz_0,j}
(z_j)}\leq M$ on $S_j$, and the result follows.
\end{proof}

\section{Strong \ases\ following a direction}

Our first objective will be to show that if $S$
is a polysector, and $f\in \OO (S)$ has null \sase\ following a
multidirection, then $f$ has null \sase\ in $S$. In one
variable, the key result is Lemma 2 from \cite{Fruzha}, which is very similar to forthcoming Lemma~\ref{asqueroso}. However, Lemma 2 there imposes some technical conditions, in particular, the
constant $C$ that appears needs to be at least 1. The
generalization to the several variables case needs to drop this
condition, but then the proof presented in \cite{Fruzha} is no longer
valid. We give in the sequel the necessary modifications.

\begin{lema} \label{cotaeninfinito}
Let $S=S(\alpha,\beta;\infty)$ be an
unbounded sector, and $f\in \OO (S)\cap {\mathcal C}(\overline{S})$ such
that $\norm{f(z)}\leq M$ for every $z\in S$, and $\norm{f(z)}\leq
\frac{C}{\norm{z}^\lambda}$ on $\arg (z)=\alpha$, for some $M,C,\lambda>0$. Then,
$$
\norm{f(z)}\leq K \cdot \Big(
\frac{C}{\norm{z}^{\lambda}}\Big)^{\frac{\beta-\theta}{\beta
- \alpha}},
$$
for every $z\in \overline{S}\setminus\{0\}$, where $\theta=\arg (z)$, $K=\max \{ 1,M\}$.
\end{lema}

\begin{proof}

The proof of Lemma 2 in \cite{Fruzha} shows how to construct a holomorphic function $h$ in $\overline{S}\setminus \{ 0\}$ such that
$$
\norm{e^{h(z)}}=e^{\Re h(z)}= \Big(
\frac{C}{\norm{z}^\lambda}\Big)^{\frac{\beta-\theta}{\beta-\alpha}}
$$
for every $z$. Indeed, one defines
$$
h(z)=-\frac{i\lambda}{2(\beta-\alpha)}(\log(z))^2+
\frac{-\lambda\beta+i\ln C}{\beta-\alpha}\,\log(z)+\frac{\beta\ln C}{\beta-\alpha}.
$$
Consider $g(z):=f(z)e^{-h(z)}$. If
$\arg (z)= \alpha$, we have
$$
\norm{g(z)}\leq \Big(
\frac{C}{\norm{z}^{\lambda}}\Big)\cdot \Big(
\frac{C}{\norm{z}^{\lambda}}\Big)^{-1}=1.
$$
If $\arg (z)=\beta$, $\norm{g(z)}\leq M$. As $g$ has
subexponential growth at infinity, \PL\ principle shows that
$\norm{g(z)}\leq K$. Then
$$
\norm{f(z)}\leq \norm{g(z)}\cdot \norm{e^{h(z)}}\leq K\cdot
\Big(
\frac{C}{\norm{z}^\lambda}\Big)^{\frac{\beta-\theta}{\beta-\alpha}}.
$$
\end{proof}

Let us now return to the origin. Given two rays $\arg(z)=\alpha$,
$\arg(z)=\beta$, a \textit{sectorial domain} of opening
$(\alpha,\beta)$ is an open set $U$, such that $0\in
\overline{U}\setminus U$, and that for each $(\alpha',\beta')$,
$\alpha<\alpha'<\beta'<\beta$, there exists $R(\alpha',\beta')>0$
such that  $S(\alpha',\beta'; R(\alpha',\beta'))\subseteq U$.
Most results concerning \ases\ on sectors can be restated, with
obvious modifications, for sectorial domains, as they are often
of local nature with respect to the origin and the radii of the
sectors involved are less important than the openings. Though we are interested in obtaining estimates on a whole bounded sector, sectorial domains will be useful in our argument.
To simplify notation, we will
consider in Lemma~\ref{asqueroso} a sector bisected by the real axis, of opening
$2\alpha<\frac{\pi}{2}$ and of radius 1, i.e., $S=S(-\alpha,
\alpha; 1)$, $0<\alpha<\frac{\pi}{4}$.

\begin{lema} \label{asqueroso}

Let $f\in \OO (S)\cap {\mathcal C}(\overline{S}\setminus\{0\})$, such that
$\norm{f(z)}\leq M$ on $S$ and $\norm{f(z)}\leq C
\norm{z}^{\lambda}$ on $\arg( z)=-\alpha$, for some $M,C,\lambda>0$. Then, given $\eps\in(0,1)$ there exists a
constant $K>0$ such that on $\overline{S}\setminus\{0\}$ we have
\begin{equation}\label{cotaepsilon}
\norm{f(z)}\leq K \cdot \left(
C\norm{z}^{\lambda}\right)^{(1-\eps)\frac{\alpha-\theta}{2\alpha}},
\end{equation}
with $\theta= \arg (z)$.
\end{lema}

\begin{proof}
Consider $F(z)=f\left(\frac{1}{z}\right)$, defined on
$$
\overline{S^{-1}}= \{ z\in \C \mid -\alpha \leq \arg (z)\leq\alpha,\
\norm{z}\geq 1\}.
$$
Denote $V= S(-\alpha, \alpha; \infty )$. For every $z\in\overline{S^{-1}}$ with $\arg(z)=\alpha$ we have that $\norm{F(z)}\leq
\frac{C}{\norm{z}^{\lambda}}$. Take $z_0=e^{i\alpha}$, and
consider $G(z)= F(z+z_0)$, defined on $\overline{V}$. On
$\arg{z}=\alpha$ we have
$\norm{G(z)}\leq \frac{C}{\norm{z+z_0}^\lambda}$ , and also $\norm{z}^\lambda \leq
\norm{z+z_0}^{\lambda}$, so $\norm{G(z)}\leq
\frac{C}{\norm{z}^\lambda}$ there.

Using Lemma \ref{cotaeninfinito} for the function $G_1(z)=\overline{G(\overline{z})}$, we deduce that
$$
\norm{G(z)}\leq K_0\cdot \Big(
\frac{C}{\norm{z}^\lambda}\Big)^{\frac{\theta(z)+\alpha}{2\alpha}}
$$
on $\overline{V}\setminus\{0\}$, where $\theta(z)=\arg(z)$ and $K_0=\max \{ 1,M\}$. Hence,
$$
\norm{F(z)}\leq K_0\cdot \Big(
\frac{C}{\norm{z-z_0}^{\lambda}}\Big)^{\frac{\theta_0(z)+\alpha}{2\alpha}}
$$
on $z_0+(\overline{V}\setminus\{0\})$, with $\theta_0(z)=\arg (z-z_0)$. Write
$$
\norm{F(z)}\leq K_0\cdot \Big(
\frac{C}{\norm{z}^{\lambda}}\Big)^{\frac{\theta_0(z)+\alpha}{2\alpha}}\cdot
\norm{\frac{z}{z-z_0}}^{\lambda\cdot\frac{\theta_0(z)+\alpha}{2\alpha}}.
$$
Consider $a>1$, to be specified later on, and $z_1= az_0$. It is clear that $z_1+\overline{V}\subseteq z_0+(\overline{V}\setminus\{0\})$, and that
$$
K_1(a):=\sup \left\{ \norm{\frac{z}{z-z_0}}: z\in z_1+\overline{V}\right\}\in(1,\infty).
$$
Since $\frac{\theta_0(z)+\alpha}{2\alpha}\in[0,1]$ for every $z\in z_1+\overline{V}$, we have
in this set that
$$
\norm{F(z)}\leq K_0\cdot K_1(a)^{\lambda} \cdot \Big(
\frac{C}{\norm{z}^\lambda}\Big)^{\frac{\theta_0(z)+\alpha}{2\alpha}}.
$$
Now, given $\eps\in(0,1)$, there exists $a>0$ such that
$$
1-\eps \leq \frac{\theta_0 (z)+\alpha}{\theta (z)+\alpha} \quad
\text{ for every } z\in z_1+\overline{V},
$$
where, as before, $\theta(z)=\arg (z)$. We impose also that $a$ is big enough so that
$\frac{C}{a^{\lambda}}<1$.
Then, as
$\norm{z}>a$ on $z_1+\overline{V}$, we obtain that
$$
\norm{F(z)}\leq K_0\cdot K_1^{\lambda}\cdot \Big(
\frac{C}{\norm{z}^{\lambda}}\Big)^{\frac{\theta(z)+\alpha}{2\alpha}\cdot
\frac{\theta_0 (z)+\alpha}{\theta (z)+\alpha}}<
K_0\cdot K_1^{\lambda}\cdot \Big(
\frac{C}{\norm{z}^\lambda}\Big)^{(1-\eps)\frac{\theta
(z)+\alpha}{2\alpha}}.
$$
So, if we consider the sectorial domain $U_1:=(z_1+V)^{-1}$, on $\overline{U}_1\setminus\{0\}$ we have
$$
\norm{f(z)}\leq K_0\cdot K_1^{\lambda}\cdot
\left(C\norm{z}^{\lambda}\right)^{(1-\eps)\frac{\alpha-\theta(z)}{2\alpha}},
$$
as desired. However, some extra effort is needed in order to conclude, since $U_1$ does not contain a sector $S(-\alpha,\alpha;\rho)$ for any $\rho>0$. Observe that the opening of $U_1$ is $(-\alpha,\alpha)$, and its boundary consists of the segment $[0,\frac{1}{a}e^{-i\alpha}]$ and an arc in $\overline{S}$ joining $\frac{1}{a}e^{-i\alpha}$ and 0. If we change in the previous argument the point $z_0$ into $\overline{z}_0=e^{-i\alpha}$, everything can be repeated word by word, obtaining a large enough value $a>0$ such that, if we consider the sectorial domain $U_2:=(ae^{-i\alpha}+V)^{-1}$, on $\overline{U}_2\setminus\{0\}$ we have
$$
\norm{f(z)}\leq K_0\cdot K_2^{\lambda}\cdot
\left(C\norm{z}^{\lambda}\right)^{(1-\eps)\frac{\alpha-\theta(z)}{2\alpha}},
$$
for a suitable constant $K_2>0$. We note that the opening of $U_2$ is again $(-\alpha,\alpha)$, but now its boundary consists of the segment $[0,\frac{1}{a}e^{i\alpha}]$ and an arc in $\overline{S}$ joining $\frac{1}{a}e^{i\alpha}$ and 0. One easily realizes that $U_1\cup U_2$ does contain a sector $T=S(-\alpha,\alpha;\rho)$ for suitable $\rho>0$, depending on~$a$. On $\overline{T}\setminus\{0\}$ we have the desired estimates for $f$, with $K=K_0(\max\{K_1,K_2\})^{\lambda}$, and they can be extended to the whole of $\overline{S}\setminus\{0\}$ by suitable enlarging the constant $K$, again in a way which ultimately depends only on $a$.
\end{proof}

\begin{obse}  \label{cotauniforme}
The reader can perform the technical modifications (basically consisting of scaling, rotation and/or ramification) to deal with the case of an arbitrary bounded sector in $\C$. Most importantly,
the constant $K$ appearing in~(\ref{cotaepsilon}) depends
only on $a$, which in turn depends on $\eps$ and, moreover, has been chosen under the requirement that
$\frac{C}{a^\lambda}<1$. So, the value of $a$, and consequently that of $K$, can be taken uniformly for every $C'<C$. This is essential in the extension of the Lemma~\ref{asqueroso}
to several variables, as follows.
\end{obse}

\begin{lema}\label{asquevv}
Let $S=S_1\times \cdots \times S_n= S(\balpha, \bbeta; \brho)$ be
a bounded polysector in $\C^n$, and $f\in \OO (S)\cap {\mathcal C}\big(
\prod_{j=1}^n (\overline{S}_j\setminus \{0\})\big)$ such that $|f(\bz)|\le M$ for every $\bz\in S$. Assume there exist $C>0$ and
$\blambda = (\lambda_1,\ldots ,\lambda_n)\in (0,\infty )^n$ such
that $\norm{f(\bz)}\leq C\cdot {\norm{\bz}}^{\blambda}$ for every $\bz\in \overline{S}$ with $\arg (\bz)=\balpha$. Then, given
$\eps\in(0,1)$ there exists $K>0$ such
that for every $\bz\in \prod_{j=1}^n (\overline{S}_j\setminus \{ 0\})$ we have
$$
\norm{f(\bz)} \leq K\cdot \left(C
\norm{\bz}^{\blambda}\right)^{(1-\eps)^n\cdot \prod_{j=1}^n
\mu_j(z_j)},
$$
where $$\mu_j (z_j)= \frac{\beta_j-\arg
(z_j)}{\beta_j-\alpha_j},\quad j=1,\ldots,n.$$
\end{lema}

\begin{proof}
For the sake of simplicity, and since the general case is treated equally, assume $n=2$. On
$\arg(\bz)=(\alpha_1,\alpha_2)$ we have $\norm{f(\bz)}\leq C\norm{z_1}^{\lambda_1} \norm{z_2}^{\lambda_2}$. Fix $z_1\in S_1$
with $\arg(z_1)=\alpha_1$, and consider the holomorphic function
$f_{z_1}(z_2):= f(z_1,z_2)$ on $S_2$. Given $\eps\in(0,1)$, there
exists $K_2>0$ such that
$$
\norm{f_{z_1}(z_2)}\leq K_2\cdot \left(C\norm{z_1}^{\lambda_1}
\norm{z_2}^{\lambda_2}\right)^{(1-\eps)\mu_2(z_2)}
$$
for every $z_2\in S_2$.
Here, the role of the constant $C$ in Lemma \ref{asqueroso} is
played by $C\norm{z_1}^{\lambda_1}$. Let us note that for every $z_1\in S_1$ we have
$C\norm{z_1}^{\lambda_1}\le C\rho_1^{\lambda_1}$. So, by Remark
\ref{cotauniforme}, the constant $K_2$ can be chosen independent of
$z_1$, as long as $\arg(z_1)=\alpha_1$.

Fix now $z_2\in S_2$. We have
$$
\norm{f(z_1,z_2)}\leq K_2
(C\norm{z_2}^{\lambda_2})^{(1-\eps)\mu_2(z_2)}
\norm{z_1}^{\lambda_1 (1-\eps)\mu_2(z_2)}\text{ \ on
}\arg(z_1)=\alpha_1.
$$
Applying again  Lemma \ref{asqueroso}, there exists a constant $K_1>0$, which, reasoning as before, can be chosen independent of $z_2$, such that
$$
\norm{f(z_1,z_2)}\leq K_1\cdot K_2^{(1-\eps)\mu_1(z_1)}\cdot
(C\norm{z_2}^{\lambda_2})^{(1-\eps)^2 \mu_1(z_1)\mu_2(z_2)}\cdot
\norm{z_1}^{\lambda_1 (1-\eps)^2 \mu_1(z_1)\mu_2(z_2)}.
$$
The result follows.
\end{proof}

Let us apply this Lemma to show that null \sase\ on a
multidirection implies null \sase\ on the polysector.

\begin{prop} \label{danulo}
Let $S\subseteq \C^n$ be a polysector, and $f\in \OO (S)$ asymptotically bounded in~$S$.
If $f$ has null \sase\ following a
multidirection $\balpha= (\alpha_1,\ldots ,\alpha_n)$ in $S$, then $f$
has null \sase\ in $S$.
\end{prop}

\begin{proof}
Fix $T=\prod_{j=1}^n T_j\prec S$ and $\bN=(N_1,\ldots ,N_n)\in \N_0^n$.
After suitably splitting or widening each $T_j$, according to whether $\alpha_j$ is a direction in $T_j$ or not, one clearly sees that it suffices to deal with the case that
$\arg(z_j)=\alpha_j$ is one of the edges of the sector $T_j$. Indeed, we will consider $T_j=S(\alpha_j,\beta_j; \rho_j)$ for some $\beta_j>\alpha_j$ and $\rho_j>0$.
Take $\delta>0$ such that $\prod_{j=1}^n S(\alpha_j,
\beta_j+\delta; \rho_j) \prec S$, and consider $V=\prod_{j=1}^n V_j$, where $V_j=S(\alpha_j,\beta_j+\delta;1)$ for $j=1,\ldots,n$. Let us define the function
\begin{equation}\label{gapartirdef}
g(z_1,\ldots,z_n)=f(\rho_1 z_1,\ldots,\rho_n z_n).
\end{equation}
The hypotheses imposed on $f$ immediately imply that $g\in \OO (V)\cap {\mathcal C}\big(\prod_{j=1}^n (\overline{V}_j\setminus \{0\})\big)$, and that there exists $M>0$ such that $|g(\bz)|\le M$ for every $\bz\in V$.
Fix $\eps\in(0,1)$, and for $j=1,\ldots,n$, take $\lambda_j\in\N_0$ verifying
\begin{equation}\label{lambdaN}
\frac{\lambda_j (1-\eps)^n \delta^n }{\prod_{j=1}^n
(\beta_j+\delta-\alpha_j)} \geq N_j.
\end{equation}
Put $\blambda=(\lambda_1,\ldots,\lambda_n)$.
Since $f$ has null \sase\ following $\balpha$, there exists $C>0$ such that
$$
|f(\bz)|\le C|\bz|^{\blambda}\qquad\textrm{ on }\ \arg(\bz)=\balpha,\ \bz\in S.
$$
Hence, for every $\bz\in\overline{V}$ with $\arg(\bz)=\balpha$ we have
$$
|g(\bz)|\le C\rho_1^{\lambda_1}\cdots\rho_n^{\lambda_n}|\bz|^{\blambda}=C_1|\bz|^{\blambda}.
$$
By Lemma~\ref{asquevv}, there exists $K>0$ such that for every $\bz\in\prod_{j=1}^n (\overline{V}_j\setminus \{0\})$,
\begin{equation}\label{cotasgV}
\norm{g(\bz)} \leq K\cdot \left(C_1
\norm{\bz}^{\blambda}\right)^{(1-\eps)^n\cdot \prod_{j=1}^n
\mu_j(z_j)},
\end{equation}
where $$\mu_j (z_j)= \frac{\beta_j+\delta-\arg
(z_j)}{\beta_j+\delta-\alpha_j},\quad j=1,\ldots,n.$$
Note that for $z_j\in S(\alpha_j,\beta_j;1)\subset V_j$ one has
$$
|z_j|< 1,\qquad \mu_j (z_j)\ge \frac{\delta}{\beta_j+\delta-\alpha_j},\qquad j=1,\ldots,n.
$$
Taking into account these inequalities, (\ref{cotasgV}) and (\ref{lambdaN}), we find that there exists $C_2>0$ such that for every $\bz\in \prod_{j=1}^n S(\alpha_j,\beta_j;1)$,
$$
|g(\bz)|\le C_2|\bz|^{\bN}.
$$
Finally, recall (\ref{gapartirdef}) to conclude that for every $\bz\in T$,
$$
|f(\bz)|=\big|g\big(\frac{1}{\rho_1}z_1,\ldots,\frac{1}{\rho_n}z_n\big)\big|\le
\frac{C_2}{\rho_1^{N_1}\cdots\rho_n^{N_n}}|\bz|^{\bN},
$$
as desired.
\end{proof}

We shall now extend this result to the Gevrey case in several
variables. The statements are similar to the corresponding ones in one variable,
and we will only sketch the proofs. It is important to have in mind what a type being 0 means for an exponentially flat function (see Remark~\ref{tipo0}).

\begin{lema} \label{exponencial}
Let $S=S(\balpha, \bbeta; \brho)$ be a polysector, with
$\beta_j-\alpha_j<\pi$ for $j=1,\ldots,n$, and $f\in \OO (S)\cap {\mathcal C} \big(
\prod_{j=1}^n (\overline{S}_j\setminus \{ 0\})\big)$ and bounded in $\prod_{j=1}^n\big(\overline{S}_j\setminus\{0\}\big)$. Let $R_j(\alpha_j), R_j(\beta_j)$, for $j\in\mathcal{N}$, be nonnegative numbers such that  for every
multidirection $\boeta= (\eta_1,\ldots ,\eta_n)$, $\eta_j\in \{
\alpha_j, \beta_j\}$, $f$ is exponentially flat of type
$\bR(\boeta)= (R_1(\eta_1),\ldots , R_n (\eta_n))$ on $\boeta$. Then,
for every multidirection $\btheta$ on $S$, $f$ is exponentially
flat of type $\bR(\btheta)= (R_1(\theta_1), \ldots ,R_n
(\theta_n))$, where $R_j(\theta_j) $ is defined as follows:
\begin{enumerate}
\item[(i)] If $R_j(\alpha_j)>0$, $R_j(\beta_j)>0$, take a circle
through the points $0$, $R_j(\alpha_j)e^{i\alpha_j}$, $R_j
(\beta_j) e^{i\beta_j}$. The straight line of angle $\theta_j$
intersects the circle in $0$, $R_j(\theta_j) e^{i\theta_j}$.
\item[(ii)] If $R_j(\alpha_j)>0$, $R_j(\beta_j)=0$, take a circle
instead through 0, $R_j (\alpha_j)e^{i\alpha_j}$, tangent to the
line of direction $\beta_j$. Analogously if $R_j(\alpha_j)=0$,
$R_j(\beta_j)>0$.
\item[(iii)] If $R_j(\alpha_j)=R_j(\beta_j)=0$, take $R_j(\theta_j)=0$.
\end{enumerate}

\end{lema}

\begin{proof}
If $\{ \alpha_j, \beta_j\}$ are in the cases (i) or (ii), let $[
0,a_j]$ be the diameter of the circle. If $\{ \alpha_j,
\beta_j\}$ are in the case (iii), take $a_j=0$. Apply \PL\ Theorem~\ref{phragmenlindelofvv}
to the function
$$
g(\bz) = f(\bz )\cdot \exp \left( \frac{a_1}{z_1}+\cdots
+\frac{a_n}{z_n}\right)
$$
to obtain the result.
\end{proof}

From this result, we deduce Watson's Lemma in several variables
for functions with null \sase\ following a multidirection.
\begin{lema}
Let $S= S(\balpha, \bbeta; \brho)$ be a polysector, $f\in \OO (S)\cap {\mathcal C} \big(
\prod_{j=1}^n (\overline{S}_j\setminus \{ 0\})\big)$ and bounded in $\prod_{j=1}^n\big(\overline{S}_j\setminus\{0\}\big)$. Suppose that for some $j_0$,
$\beta_{j_0}-\alpha_{j_0}\geq \pi$, and $f$ is exponentially flat
of type $\bR= (R_1,\ldots ,R_n)\in[0,\infty)^n$ in multidirection
$\btheta_0$, with $R_{j_0}>0$. Then, $f$ identically vanishes.
\end{lema}

\begin{proof}
Fix $\bz$ with $\arg
(\bz)=\btheta_0=(\theta_1,\ldots,\theta_n)$, and consider $f_{\bz, j_0} :
S(\alpha_{j_0}, \beta_{j_0}; \rho_{j_0})\rightarrow \C$ defined
as $f_{\bz, j_0}(z)= f(\bz_{j\,'}, z)$. Applying the one variable version of Proposition~\ref{propoexpoplana}, we see that $f_{\bz, j_0}$ has null \sase\ following $\theta_{j_0}$, of Gevrey order 1. By Lemma 5 of \cite{Fruzha},
$f_{\bz,j_0}\equiv 0$. As $\bz$ was arbitrarily chosen, we deduce that $f\equiv 0$ on $\arg
(\bz)= \btheta_0$. Classical results of complex variables show
that, then, $f\equiv 0$ (see, for instance, \cite[p.\ 31]{Chabat}, where
the multidirection $\btheta_0$ is real).
\end{proof}

Combining previous results, we obtain:

\begin{prop}\label{58}
Let $S= S(\balpha, \bbeta; \brho)$ and $f\in \OO
(S)$, asymptotically bounded in $S$ and having null \sase\ $\boldsymbol{1}$-Gevrey of
type $\bR= (R_1,\ldots , R_n)\in(0,\infty)^n$ following a
multidirection $\btheta_0$ of $S$. Then,  $f$ has null \sase\
1-Gevrey in $S$.
\end{prop}

\begin{obse}\label{59}
In the previous Proposition, let us assume that
$\beta_j-\alpha_j<\pi$ for every $j$. As indicated in~\cite{Fruzha}, the type of the \ase\ can be computed
using Lemma \ref{exponencial}, particularly case 2. Geometric
considerations show that for a multidirection $\btheta =
(\theta_1,\ldots ,\theta_n)$ on $S$, the type is $\bR(\btheta)= (R_1(\theta_1), \ldots ,R_n (\theta_n))$, where
$$
R_j(\theta_j)=\begin{cases}
R_j \cdot \frac{\sin (\theta_j-\beta_j)}{\sin
(\theta_{0j}-\beta_j)} & \text{ if }\theta_j\in [\theta_{0j},
\beta_j], \\\noalign{\vskip.5em}
R_j \cdot \frac{\sin (\theta_j-\alpha_j)}{\sin
(\theta_{0j}-\alpha_j)} & \text{ if }\theta_j\in
[\alpha_j,\theta_{0j}].
\end{cases}
$$
\end{obse}

Let us now establish a Borel--Ritt--Gevrey type theorem in several
variables. This result is presented in \cite{Haraoka}. As we need
to control not only the Gevrey order, but also the type of the expansion,
we will repeat here the main steps of the proof.

\begin{defi}\label{defiserie1GtipoR}
Let $\hat{f}= \sum\limits_{\bN\in \N_0^n} f_{\bN}\bz^{\bN}\in \C
[[ \bz]]$, $\bR= (R_1,\ldots ,R_n)\in (0,\infty )^n$. We will
say that $\hat{f}$ is a ${\mathbf 1}$-Gevrey series of type $\bR$
if for every $\delta>0$, there exists $C=C(\delta)>0$ such that
$$
\norm{f_{\bN}}\leq C \cdot \bN !\cdot \prod_{j=1}^n \left(
\frac{1}{R_j}+\delta
\right)^{N_j}.
$$
\end{defi}

\begin{obse}\label{obseFA}
For a function $f$ admitting $\boldsymbol{1}$-Gevrey \sase\ in a polysector $S$, consider the formal series $\textrm{FA}(f):=\sum_{\bN\in\N_0^n}f_{\bN}\bz^{\bN}$, whose coefficients are the constant elements of $\textrm{TA}(f)$ given by
$$
f_{\bN}=\lim_{\scriptstyle \bz\to \textbf{0}\atop \scriptstyle \bz\in T\prec S}\frac{D^{\bN}f(\bz)}{\bN!}.
$$
Then, one has that $\textrm{FA}(f)$ is a ${\mathbf 1}$-Gevrey series. The next paragraph and theorem show us a way to go in the opposite direction.
\end{obse}

For a ${\mathbf 1}$-Gevrey series $\hat{f}= \sum\limits_{\bN\in \N_0^n} f_{\bN}\bz^{\bN}$ of type $\bR$, define its ${\mathbf 1}$-Borel
transform as $\varphi (\bz)=\sum\limits_{\bN\in \N_0^n}
\frac{1}{\bN !} f_{\bN} \bz^{\bN}$, convergent in $D({\mathbf
0}, \bR) $. Take $\bz_0= (z_{01}, \ldots z_{0n})\in D({\mathbf
0}, \bR)$, and define the truncated Laplace transform as
\begin{equation} \label{BL}
F(\bz)=\mathcal{L}_{\bz_0}^T(\varphi)(\bz)=\frac{1}{z_1\cdots z_n}\cdot \int_{0}^{z_{01}}\cdots
\int_0^{z_{0n}} \varphi (\bt )\cdot \exp \left(
-\frac{t_1}{z_1}-\ldots -\frac{t_n}{z_n}\right) dt_1\cdots dt_n.
\end{equation}
$F(\bz)$ is holomorphic on $S_0= \mathop{\prod}\limits_{j=1}^n
S(\arg (z_{0j}) - \frac{\pi}{2}, \arg (z_{0j})+\frac{\pi}{2};\infty)$.

A total family of \sase\ may be defined from the series $\hat{f}$ as
follows. Let $\emptyset\neq J\subseteq {\mathcal N}$. Write
$$
\hat{f}(\bz) = \sum_{\balpha_J\in \N_0^J} f_{\balpha_J}
(\bz_{J'})\cdot \bz_J^{\balpha_J},
$$
where $f_{\balpha_J} (\bz_{J'})= \sum\limits_{\balpha_{J'}\in\N_0^{J'}}
f_{(\balpha_J, \balpha_{J'})} \bz_{J'}^{\balpha_{J'}}$ are
${\mathbf
1}_{J'}$-Gevrey series. Define
$$
\varphi_{\balpha_{J}}(\bz_{J'})= \sum_{\balpha_{J'}\in \N_0^{J'}}
\frac{1}{\balpha_{J'}!} \cdot f_{(\balpha_{J}, \balpha_{J'})}\cdot
\bz_{J'}^{\balpha_{J'}},
$$
and
$$
F_{\balpha_{J}}(\bz_{J'})=\frac{1}{\bz_{J'}}\cdot \left(
\prod_{j\in J'} \int_0^{z_{0j}}\right) \varphi_{\balpha_J }
(\bt_{J'})\exp \left( -\sum_{j\in J'} \frac{t_j}{z_j}\right)\cdot
\prod _{j\in J'} dt_j
$$
(where, for short, the product symbol before the integrals denotes an iterated integral). Consider ${\mathcal F}=
{\mathcal F}(\hat{f}) = \{ F_{\balpha_{J}}(\bz_{J'}); \emptyset
\neq J \subseteq \mathcal{N},\ \balpha_J \in\N_0^{J}\}$.

\begin{teor}[Borel-Ritt-Gevrey] \label{BRG}
With previous notations, $F(\bz)$ has the family ${\mathcal
F}(\hat{f})$ as total family of ${\mathbf 1}$-Gevrey \sase\ in
$S_0$. Moreover, given a multidirection~$\btheta$ in~$S_0$, the type
of this \sase\ is $\bR (\btheta) = (R_1(\theta_1), \ldots
,R_n(\theta_n))$, where $R_j(\theta_j)= \norm{z_{0j}}\cdot \cos
(\theta_j-\arg (z_{0j}))$.
\end{teor}

\begin{proof}

We shall only compute the type of the \sase\ on every multidirection, referring to~\cite{Haraoka} for a more detailed account on some of the forthcoming computations. Indeed, given $\bN\in \N_0^n$, we depart from the following expression in~\cite[p.\ 376]{Haraoka}:

\begin{multline*}
F(\bz)-\hbox{App}_{\bN}({\mathcal F})(\bz) \\
=\frac{1}{|\bz|^{\boldsymbol{1}}} \sum_{\emptyset\neq J\subseteq\mathcal{N}}(-1)^{\# J}
\sum_{\balpha_J<\bN_J} \sum_{\balpha_{J'}\geq \bN_{J'}}
\frac{f_{(\balpha_{J},\balpha_{J'})}}{\balpha_{J}!\balpha_{J'}!} \left( \prod_{j\in J}
\int_{z_{0j}}^{\infty} t_j^{\alpha_j} e^{-\frac{t_j}{z_j}}d
t_j \right) \left( \prod_{j\in J'}\int_{0}^{z_{0j}} t_j^{\alpha_j}
e^{-\frac{t_j}{z_j}}dt_j \right).
\end{multline*}

Given $\delta>0$, $\exists\, C=C(\delta)>0$ such that, for every
$\balpha=(\a_1,\ldots,\a_n) \in \N_0^n$,
$$\norm{f_{\balpha}}\leq C\balpha !
\prod_{j=1}^n  \left( \frac{1}{R_j}+\delta\right)^{\alpha_j}.$$
So, we can estimate
\begin{multline*}
\norm{F(\bz)-\hbox{App}_{\bN}({\mathcal F})(\bz)} \\
\leq C\cdot\frac{1}{|\bz|^{\boldsymbol{1}}} \sum_{\emptyset\neq J\subseteq {\mathcal
N}} \sum_{\balpha_J<N_{J}} \sum_{\balpha_{J'}\geq N_{J'}}
\prod_{j\in J}
\norm{\int_{z_{0j}}^{\infty }
\left(\left(\frac{1}{R_j}+\delta\right)
t_j\right)^{\alpha_j} e^{-\frac{t_j}{z_j}} dt_j}  \\ \cdot
\norm{\prod_{j\in J'} \int_{0}^{z_{0j}} \left(
\left(\frac{1}{R_j}+\delta\right) t_j\right)^{\alpha_j}
e^{-\frac{t_j}{z_j}}dt_j}
\end{multline*}
\begin{multline*}
\leq  C\frac{1}{|\bz|^{\boldsymbol{1}}} \sum_{\emptyset\neq J\subseteq {\mathcal
N}} \sum_{\balpha_{J}<N_J} \sum_{\balpha_{J'}\geq N_{J'}}
\prod_{j\in J} \left( \frac{1}{R_j}+\delta\right)^{\alpha_j}
\norm{z_{0j}}^{\alpha_j+1} \\  \cdot \norm{\int_{1}^{\infty }
s_j^{N_j}e^{-s_j\cdot \frac{z_{0j}}{z_j}} ds_j} \cdot
\prod_{j\in J'} \left( \frac{1}{R_j}+\delta\right)^{\alpha_j}
\norm{z_{0j}}^{\alpha_j+1} \cdot \norm{\int_0^1 s_j^{N_j}
e^{-s_j \frac{z_{0j}}{z_j}} ds_j}
\end{multline*}
$$\leq  C\frac{1}{|\bz|^{\boldsymbol{1}}} \sum_{\emptyset\neq J\subseteq {\mathcal
N}} \sum_{\balpha_{J}<N_J} \sum_{\balpha_{J'}\geq N_{J'}}
\frac{1}{\cos (\btheta - \arg(\bz_0))^{\bN+{\mathbf 1}}} \cdot
\frac{\bN !\cdot \norm{\bz}^{\bN +{\mathbf
1}}}{\norm{\bz_0}^{\bN +{\mathbf 1}}} \cdot  \prod_{j=1}^n
\left( \frac{1}{R_j}+\delta\right)^{\alpha_j}
\norm{z_{0j}}^{\alpha_j+1}.
$$

As $\norm{z_{0j}}<R_j$, taking $\delta$ small enough, we obtain a
bound
$$
D\cdot \bN ! \cdot \frac{1}{\bR(\btheta )^{\bN}} \cdot
\norm{\bz}^{\bN},
$$
for certain $D=D(\delta)>0$, as desired.
\end{proof}

\begin{obse}
Regarding the previous result, it makes no difference to consider series, resp.\ functions, with coefficients, resp.\ values, in a complex Banach space, instead of limiting our construction to complex series and functions.
\end{obse}

In the sequel, we will need to consider vector spaces of
holomorphic functions in a polysector admitting ${\boldsymbol{1}}$-Gevrey \ase\ with a
type depending on the multidirection. In order to guarantee that these spaces have a suitable Banach space structure, their definition will involve estimates for the derivatives.

\begin{defi}\label{defW1RSE}
Let $E$ be a complex Banach space, $S=S(\balpha,\bbeta;\brho)\subseteq \C^n$ a polysector,
and, for every $j=1,\dots,n$, let $R_j:(\alpha_j,\beta_j)\to(0,\infty)$ be a function such that for every compact subset $K\subset(\alpha_j,\beta_j)$ one has that \begin{equation}\label{condicionRj}
\inf_{\theta\in K}R_j(\theta)>0.
\end{equation}
For every  multidirection $\btheta=(\theta_1,\ldots,\theta_n)$ in $S$ we write
$$
\bR(\btheta)=\big(R_1(\theta_1),\ldots,R_n(\theta_n)\big).
$$
We will denote by
$\mathcal{W}_{\bR(\btheta)}^{1} (S,E)$
the space of holomorphic functions from $S$ to $E$ verifying
\begin{equation}\label{normaW1RSE}
\norma{f}_{\bR(\btheta)}:= \sup \{ \frac{\norma{D^{\bN}f
(\bz)
} \bR(\btheta(\bzeta))^{\bN} }{\bN !^2} : \btheta(\bz)=\arg (\bz),\ \bz\in
S,\ \bN\in \N_0^n
\} <+\infty .
\end{equation}
\end{defi}

\begin{prop}
The pair
$\left( \mathcal{W}_{\bR(\btheta)}^{1} (S,E),
\norma{\cdot}_{\bR(\btheta)}
\right)$
is a Banach space.
\end{prop}

\begin{proof}
Thanks to the conditions (\ref{condicionRj}) imposed on the functions $R_j$, it is easy to prove that a Cauchy sequence in $\mathcal{W}_{\bR(\btheta)}^{1} (S,E)$ converges uniformly on the compact subsets of $S$ to an \textit{a fortiori\/} holomorphic function from $S$ to $E$, which will ultimately belong to $\mathcal{W}_{\bR(\btheta)}^{1} (S,E)$.
\end{proof}

\begin{obses}\label{obseW1RSE}
\linea
\begin{enumerate}
\item[(i)] Suppose that for $j=1,\ldots,n$, we are given functions $\widetilde{R}_j:(\alpha_j,\beta_j)\to(0,\infty)$ verifying (\ref{condicionRj}) and $\widetilde{R}_j(\theta_j)\leq R_j(\theta_j)$ for every $\theta_j\in(\alpha_j,\beta_j)$. Put $\widetilde{\bR}=(\widetilde{R}_1,\ldots,\widetilde{R}_n)$. Then, the inclusion ${\mathcal
W}^1_{\bR(\btheta)} (S,E)\subseteq {\mathcal W}^1_{\widetilde{\bR}(\btheta)}
(S,E) $ is continuous.
\item[(ii)] As follows from Proposition 3 in \cite{Haraoka}, given a function $f$ with ${\boldsymbol{1}}$-Gevrey \sase\ in $S$
(as stated in Definition~\ref{defsase}), for every $T\prec S$ there exists a constant vector $\bA\in(0,\infty)^n$ such that $f\in{\mathcal W}^1_{\bA}(T,E)$.\par
Conversely, every $f\in {\mathcal W}^1_{\bR(\btheta)}(S,E)$ admits ${\boldsymbol{1}}$-Gevrey
\sase\ in~$S$. Indeed, we may associate to $f$ a coherent family
${\mathcal F}$ by the expressions in~(\ref{familiatotal}), and for $\bN=(N_1,\ldots,N_n)\in \N_0^n$ we have that
\begin{multline*}
f(\bz) - \mbox{App}_{\bN} ({\mathcal F})(\bz) \\
=\prod_{\sst j=1\atop\sst N_j\neq 0}^n\left(
\int_0^{z_j}\,dt_{j,1}\int_0^{t_{j,1}}\,dt_{j,2}
\cdots\int_0^{t_{j,N_j-1}}\,dt_{j,N_j}\right)
D^{\bN} f(t_{1,N_1},t_{2,N_2},\ldots,t_{n,N_n})
\end{multline*}
for every $\bz=(z_1,\ldots,z_n)\in S$, where $t_{k,0}=z_k$ and the product symbol stands for an iterated integral (see \cite{Haraoka,jesusisla}).
In order to conclude, it suffices to estimate this integral on every $T\prec S$ by using~(\ref{normaW1RSE}) and~(\ref{condicionRj}). In this way, one also gets that every $f\in {\mathcal W}^1_{\bR(\btheta)}(S,E)$ admits ${\boldsymbol{1}}$-Gevrey \sase\ in every multidirection $\btheta$ in~$S$ with type $\bR(\btheta)$.
\end{enumerate}
\end{obses}

Our next aim is to show that the function studied in Theorem~\ref{BRG} belongs to a suitable space of this kind.

\begin{prop}\label{transLaplaW1Rtilde}
Let $E$ be a complex Banach space and $\hat{f}= \sum\limits_{\bN\in \N_0^n} f_{\bN}\bz^{\bN}\in E[[ \bz]]$ be a ${\mathbf 1}$-Gevrey series of type $\bR= (R_1,\ldots ,R_n)\in (0,\infty )^n$. Then, the function $F$ constructed in Theorem~\ref{BRG} belongs to ${\mathcal W}^1_{\widetilde{\bR}(\btheta)}(S_0,E)$,
where the function
$$
\widetilde{\bR}(\btheta)=(\widetilde{R}_1(\theta_1),\ldots,\widetilde{R}_n(\theta_n))
$$
is such that for every multidirection~$\btheta=(\theta_1,\ldots,\theta_n)$ in~$S_0$,
$$
\frac{1}{4.7}\norm{z_{0j}}\cos^2
(\theta_j-\arg (z_{0j}))\le \widetilde{R}_j(\theta_j)\le \frac{1}{2}\norm{z_{0j}} \cos^2
(\theta_j-\arg (z_{0j})),\quad j=1,\ldots,n.
$$
\end{prop}

\begin{proof}
It makes no difference to take $E=\C$.
Put $\bzeta_0=(z_{01},\ldots,z_{0n})$, $\arg(\bz_0)=(\theta_{01},\ldots,\theta_{0n})$.
For every multidirection $\btheta=(\theta_1,\ldots,\theta_n)$ in $S_0$ we write $\bdelta(\btheta)=\cos(\btheta-\arg(\bzeta_0))$. Given $\bz=(z_1,\ldots,z_n)\in S_0$ with argument $\btheta$, and  $c\in(0,1)$, one easily checks that the closed polydisc centered at $\bzeta$ with polyradius
$$
c\bdelta(\btheta)|\bzeta|:=(c\cos(\theta_1-\theta_{01})|z_1|,\ldots, c\cos(\theta_n-\theta_{0n})|z_n|)
$$
is contained in $S_0$. Let $\bN=(N_1,\ldots,N_n)\in\n_0^n$.
Since $D^{\bN}\textrm{App}_{\bN}(\textrm{TA}(F))\equiv 0$ on $S_0$,
Cauchy's integral formula allows us to write
$$
D^{\bN}F(\bzeta)=
\frac{\bN!}{(2\pi i)^n}
\int_{\ds\partial_d D(\bzeta,c\bdelta(\btheta)|\bzeta|)}
\frac{F(\bomega)-\textrm{App}_{\bN}(\textrm{TA}(F))(\bomega)}
{(\bomega- \bzeta)^{\bN+{\bf 1}}}\,d \bomega,
$$
where $\partial_d$ stands for the distinguished boundary.
According to Theorem~\ref{BRG}, there exists
$D>0$ such that
$$
\big|F(\bomega)-\textrm{App}_{\bN}(\textrm{TA}(F))(\bomega)|\le
D\cdot \bN ! \cdot \frac{1}{\bR(\bomega)^{\bN}} \cdot
\norm{\bomega}^{\bN},\quad\bomega=(\omega_1,\ldots,\omega_n)\in S_0,
$$
where $\bR (\bomega) = (\norm{z_{01}}\cos(\arg(\omega_1)-\theta_{01}),\ldots,
\norm{z_{0n}}\cos(\arg(\omega_n)-\theta_{0n}))$. Hence,
\begin{multline*}
|D^{\bN}F(\bzeta)|\le D\bN!^2\frac{(\boldsymbol{1}+c\bdelta(\btheta))^{\bN}}{(c\bdelta(\btheta))^{\bN}} \max_{\bomega\in \partial_d D(\bzeta,c\bdelta(\btheta)|\bzeta|)} \frac{1}{\bR(\bomega)^{\bN}}\\
\le D\bN!^2\prod_{j=1}^n\Big(\frac{1+c\cos(\theta_j-\theta_{0j})}{c\cos(\theta_j-\theta_{0j})} \max_{|\omega_j-z_j|=c\cos(\theta_j-\theta_{0j})|z_j|} \frac{1}{\norm{z_{0j}}\cos(\arg(\omega_j)-\theta_{0j})}\Big)^{N_j}.
\end{multline*}
Now, it is clear that we may reason componentwise, so we assume $n=1$ and drop the subindices. Put $\delta(\theta)=\cos(\theta-\theta_0)$. One finds from elementary considerations that
\begin{multline*}
\max_{|\omega-z|=c\delta(\theta) |z|} \frac{1}{\norm{z_{0}}\cos(\arg(\omega)-\theta_{0})}
=
\frac{1}{|z_0|\cos\big(|\theta-\theta_0|+\arcsin(c\delta(\theta))\big)}\\
=\frac{1}{|z_0|\delta(\theta) \big(\sqrt{1-c^2\delta(\theta)^2}-c\sqrt{1-\delta(\theta)^2}\,\big)}.
\end{multline*}
Since $c\in(0,1)$ was arbitrary, one obtains that
$$
|D^NF(z)|\le DN!^2\Big(\frac{1}{\norm{z_{0}} \delta(\theta)^2}
\inf_{c\in(0,1)}\frac{1+c\delta(\theta)}{c\big(\sqrt{1-c^2\delta(\theta)^2}- c\sqrt{1-\delta(\theta)^2}\,\big)}\Big)^N,
$$
what means that $F\in{\mathcal W}^1_{\widetilde{R}(\theta)}(S_0)$ for
\begin{equation}\label{tildeRg}
\widetilde{R}(\theta)=\norm{z_{0}} \delta(\theta)^2\sup_{c\in(0,1)}\frac{c\big(\sqrt{1-c^2\delta(\theta)^2}- c\sqrt{1-\delta(\theta)^2}\,\big)}{1+c\delta(\theta)}=\norm{z_{0}} \delta(\theta)^2g(\delta(\theta)).
\end{equation}
Finally, we estimate the value of $g(\delta(\theta))$ (in particular, we deduce that $\widetilde{R}$ verifies~(\ref{condicionRj})). On one hand, observe that the numerator and denominator in the last quotient increase with $\delta(\theta)\in(0,1]$, so that
$$
g(\delta(\theta))\le \sup_{c\in(0,1)}c\sqrt{1-c^2}=\frac{1}{2}.
$$
On the other hand, taking $c=1/2$ and with the help of calculus software we obtain that
$$
g(\delta(\theta))\ge \frac{2\sqrt{1-\delta(\theta)^2/4}- \sqrt{1-\delta(\theta)^2}}{4+2\delta(\theta)}\ge \frac{1}{4.7}.
$$
\end{proof}


The following result can be essentially found in
\cite{jesusisla,javi2}, and we omit its eminently technical proof.
It will be useful when reducing our problem in several variables to a one-variable situation.

\begin{lema} \label{lemab}
Let $S=S_1\times S_2$ be a product of two polysectors, and let $\bR_1,\bR_2$ be functions as those described in Definition~\ref{defW1RSE}, respectively acting on the multidirections $\btheta_1$ on $S_1$ and $\btheta_2$ on $S_2$. Then:
\begin{enumerate}
\item[(i)]  The map
$$\Psi_1:\mathcal{W}_{\bR_1(\btheta_{1})}^{1}(S_{1},
\mathcal{W}_{\bR_{2}(\btheta_{2})}^{1}(S_{2}))\to
\mathcal{W}_{(\bR_{1}(\btheta_{1}),\bR_{2}(\btheta_{2}))}^{1}(S_{1}\times S_{2})$$
defined by $\Psi_1(f)(\bz_{1},\bz_{2})=f(\bz_{1})(\bz_{2})$ is well
defined and it is an isomorphism.
Moreover, for each $n,m\in\N_{0}$ we have
$$D^{(n,m)}\Psi_1(f)(\bz_1,\bz_2)=D^{m}(D^{n}f(\bz_1))(\bz_2),\qquad
(\bz_{1},\bz_{2})\in S_1\times S_2.$$
\item[(ii)] The map
 $$\Psi_2:\mathcal{W}_{\bR_{1}(\btheta_{1})}^{1}(S_{1},
\mathcal{W}_{\bR_{2}(\btheta_{2})}^{1}(S_{2}))\to\mathcal{W}_{\bR_{2}(
\btheta_{2})}^{1}(S_{2},\mathcal{W}_{\bR_{1}(\btheta_{1})}^{1}(S_{1}))$$
defined by
 $\Psi_2(f)(\bz_{2})(\bz_{1})=f(\bz_{1})(\bz_{2})$ is an isomorphism.
Moreover, for each
$n,m\in\N_{0}$ we have that
$$D^{n}(D^{m}\Psi_2(f)(\bz_2))(\bz_1)=
D^{m}(D^{n}f(\bz_1))(\bz_2),\qquad (\bz_{1},\bz_{2})\in S_1\times S_2.$$
\end{enumerate}
\end{lema}

In order to state our main result, we need a definition.

\begin{defi}\label{deffamiliaprimerorden}
Let ${\mathcal F}= \left\{ f_{\bN_J}: \emptyset \neq
J\subseteq {\mathcal N},\ \bN_J \in \N_0^J\right\}$ be a coherent
family on a polysector $S=S_1\times \cdots \times S_n$. The
subfamily ${\mathcal F}_1$ consisting of the
elements of ${\mathcal F}$ that depend on $n-1$ variables is called the first order family associated to ${\mathcal F}$. More precisely,
$$
{\mathcal F}_1 = \left\{ f_{m_{\{ j \} }} (\bz_{{j\,}'}): \ j\in {\mathcal N},\ m\in \N_0\right\},
$$
where, according to (\ref{familiatotal}),
$$
f_{m_{\{ j \} }} (\bz_{{j\,}'})=\lim_{z_j\to 0}\frac{D^{m\be_j}f(z_j,\bz_{{j\,}'})}{m!}.
$$
For simplicity, we will write $f_{jm}$ instead of $f_{m_{\{ j \} }}$. Note that ${\mathcal F}_1$ consists of the $n$ sequences $\{f_{jm}\}_{m\in\N_0}$, $j=1,\ldots,n$.

Given $f$ admitting \sase\ in $S$, $\textrm{TA}_1(f)$ will denote the first order
family associated to $\textrm{TA}(f)$. $\textrm{TA}_1(f)$ inherits coherence conditions from $\textrm{TA}(f)$ (see Definition~\ref{deficohe}). We say a first order family ${\mathcal F}_1$ is coherent if it verifies such conditions.
\end{defi}

\begin{obse}
There is a one-to-one correspondence between coherent families and coherent first order families (see~\cite{jesusisla,javifelix}).
\end{obse}

\begin{ejem}
If $n=2$, a coherent family is
$${\mathcal F}=
\{ f_{1n}(z_2),\  f_{2m}
(z_1),\  f_{nm}:\ n,m\in\N_0\} ,$$
where the coherence conditions mean that the \ase\ of $f_{1n}(z_2)$ is $\sum_{m} f_{nm} z_2^m$, and the \ase\ of
$f_{2m} (z_1) $ is $\sum_{n} f_{nm} z_1^n$, for every $n,m\in\N_0$. The first order associated family is, loosely speaking,
$${\mathcal F}_1=
\{ f_{1n}(z_2)\}_{n\in\N_0}\cup \{ f_{2m}(z_1)\}_{m\in\N_0}.
$$
\end{ejem}

Let us now introduce some new spaces needed in the sequel.

\begin{defi}
Let $E$ be a Banach space, and $\ba= (a_{\bN})_{\bN\in \N_0^n}
$ a multisequence in $E$. Fix $\bA\in (0,\infty )^n$. We will say
that $\ba\in \Gamma^1_{\bA} (\N_0^n, E)$ if
$$
\norm{\ba}_{\bA}:=\sup \left\{
\frac{\norma{a_{\bN}}_E\bA^{\bN}}{\bN !}: \bN\in
\N_0^{n}\right\} <+\infty.
$$
The pair $\left( \Gamma^1_{\bA}(\N_0^n, E),
\norm{\cdot}_{\bA}\right)$ is a Banach space.
\end{defi}

\begin{obses}
\linea
\begin{enumerate}
\item[(i)] As before, given $\bA,\bB\in(0,\infty)^n$ with $\bA\leq \bB$, we have that the inclusion $\Gamma^1_{\bB}
(\N_0^n, E)\subseteq \Gamma^1_{\bA} (\N_0^n, E)$ is continuous.
\item[(ii)] Given $f\in\mathcal{W}_{\bR(\btheta)}^{1} (S,E)$, it is clear (see Remark~\ref{obseFA}) that $\textrm{FA}(f)\in\Gamma^1_{\bR(\btheta)} (\N_0^n, E)$ for every multidirection $\btheta$ in $S$.
\item[(iii)] If $\ba= (a_{\bN})_{\bN\in \N_0^n}\in \Gamma^1_{\bA} (\N_0^n, E)$, then
$\hat{f}=\sum_{\bN\in\N_0^n}a_{\bN}\bz^{\bN}$ is a $\boldsymbol{1}$-Gevrey series of type~ $\bA$.
\end{enumerate}
\end{obses}

We are ready to state the following result, which tells us which are the natural conditions to depart from in an interpolation problem in the spaces $\mathcal{W}_{\bR(\btheta)}^{1}(S)$. Its proof follows directly from the definitions of the different spaces involved.

\begin{lema}\label{lemafamiprimorde}
Let $S=\prod_{j=1}^n S_j$ be a polysector, and suppose $f\in {\mathcal W}_{\bR(\btheta)}^1(S)$ for some function
$$
\bR(\btheta)=(R_1(\theta_1),\ldots,R_n(\theta_n)).
$$
Let
$$\textrm{TA}_1(f)=
\{ f_{jm}(\bz_{{j\,}'}): j\in\mathcal{N},\ m\in\N_0\}
$$
be the first order family of \sase\ of $f$ (see Remarks~\ref{obseW1RSE}). Then we have:
\begin{enumerate}
\item[(i)] For $j=1,\ldots,n$ and $m\in\N_0$, $f_{jm}\in{\mathcal W}_{\bR_{{j\,}'}(\btheta_{{j\,}'})}^1 (S_{{j\,}'})$.
\item[(ii)] For $j=1,\ldots,n$, $\{ f_{jm}\}_{m\in \N_0}\in \Gamma^1_{R_j(\theta_j)} (\N_0,
{\mathcal W}_{\bR_{{j\,}'}(\btheta_{{j\,}'})}^1 (S_{{j\,}'}))$ for every direction $\theta_j$ on $S_j$.
\end{enumerate}
\end{lema}

\begin{ejem}
Suppose $n=2$ and $f\in {\mathcal W}_{\bR(\btheta)}^1(S)$, where
$S=S_1\times S_2$, $\bR=(R_1,R_2)$. Let
$$
\textrm{TA}_1(f)=\{ f_{1n}(z_2),\ f_{2m}(z_1):n,m\in\N_0\}.
$$
By the previous Lemma, $\{ f_{1n}\}_{n\in \N_0}\in \Gamma^1_{R_1(\theta_1)} (\N_0,
{\mathcal W}^1_{R_2(\theta_2)} (S_2))$ for every direction $\theta_1$ on $S_1$, and
$\{ f_{2m}\}_{m\in \N_0}\in \Gamma^1_{R_2(\theta_2)} (\N_0,
{\mathcal W}^1_{R_1(\theta_1)} (S_1))$ for every direction $\theta_2$ on $S_2$.
\end{ejem}

The next interpolation Lemma is essentially stated in
\cite[Thm.\ 3.4]{javi2}, though that result did not contain the specific estimates we are going to provide now.

\begin{lema}\label{521}
Let $S=S(\balpha,\bbeta;\infty)=\prod_{j=1}^nS_j\subseteq \C^n$ be a polysector such that, for every $j$, the opening of $S_j$, $\beta_j-\alpha_j$, is less than or equal to $\pi$.
For $j=1,\dots,n$, let $R_j:(\alpha_j,\beta_j)\to(0,\infty)$ be a function under the conditions in Definition~\ref{defW1RSE}, and define $\bR=(R_1,\ldots,R_n)$ as usual.
Suppose given a coherent first order family ${\mathcal F}_1=\{f_{jm}:j=1,\dots,n,\ m\in\N_0\}$ such that for every $j=1,\ldots,n$ and every direction $\theta_j$ on $S_j$,
$$
\{ f_{jm}\}_{m\in \N_0}\in \Gamma^1_{R_j(\theta_j)} (\N_0,
{\mathcal W}_{\bR_{{j\,}'}(\btheta_{{j\,}'})}^1 (S_{{j\,}'})).
$$
For every $j=1,\ldots,n$, choose $z_{0j}$ in the bisecting direction $\theta_{0j}$ of $S_j$ and such that $|z_{0j}|<\sup_{\theta_j\in(\alpha_j,\beta_j)}R_j(\theta_j)$,
and consider the functions $\widetilde{R}_j\colon(\alpha_j,\beta_j)\to(0,\infty)$ introduced in Proposition~\ref{transLaplaW1Rtilde}. If we put
\begin{equation}\label{hatR}
\widehat{R}_j=\min(R_j,\widetilde{R}_j),\qquad j=1,\ldots,n,
\end{equation}
and $\widehat{\bR}=(\widehat{R}_1,\ldots,\widehat{R}_n)$,
then there exists $f\in{\mathcal W}_{\widehat{\bR}(\btheta)}^1 (S)$
such that $\textrm{TA}_1(f)=\mathcal{F}_1$.
\end{lema}

\begin{proof}
For simplicity of the notations, the proof will be carried on for the case $n=2$, the general case being treated in the same way. So, we depart from
a coherent first order family
$$
\mathcal{F}_1=\{ f_{1n}(z_2),\ f_{2m}(z_1):n,m\in\N_0\},
$$
with $\{ f_{1n}\}_{n\in \N_0}\in \Gamma^1_{R_1(\theta_1)} (\N_0,
{\mathcal W}^1_{R_2(\theta_2)} (S_2))$ for every direction $\theta_1$ on $S_1$, and
\begin{equation}\label{f2m}
\{ f_{2m}\}_{m\in \N_0}\in \Gamma^1_{R_2(\theta_2)} (\N_0,
{\mathcal W}^1_{R_1(\theta_1)} (S_1))
\end{equation}
for every direction $\theta_2$ on $S_2$. The series $\varphi_1(z_1)=\!\sum_{n\in\N_0}f_{1n}z_1^n/n!$, with coefficients in ${\mathcal W}^1_{R_2(\theta_2)} (S_2)$, converges. For $z_{01}$ as stated, put $S_{01}=S(\theta_{01}-\pi/2,\theta_{01}+\pi/2;\infty)\supset S_1$. We may consider the function $H_1^{[1]}:=\mathcal{L}_{z_{01}}^T(\varphi_1)$ which, by Proposition~\ref{transLaplaW1Rtilde}, belongs to
${\mathcal W}^1_{\widetilde{R}_1(\theta_1)} (S_{01},{\mathcal W}^1_{R_2(\theta_2)} (S_2))$ and admits the series $\sum_{n\in\N_0}f_{1n}z_1^n$ as its asymptotic expansion in $S_{01}$. According to Lemma~\ref{lemab}, the function $H^{[1]}:=\Psi_1(H_1^{[1]})$ belongs to
$${\mathcal W}^1_{(\widetilde{R}_1(\theta_1),R_2(\theta_2))} (S_{01}\times S_2)\subset {\mathcal W}^1_{(\widetilde{R}_1(\theta_1),R_2(\theta_2))} (S_1\times S_2).$$
Put $\textrm{TA}_1(H^{[1]})=\{h_{1n}^{[1]},h_{2m}^{[1]}:n,m\in\N_0\}$.
From Lemma~\ref{lemab} we also get that $h_{1n}^{[1]}=f_{1n}$ for every $n\in\N_0$, while Lemma~\ref{lemafamiprimorde} implies that
$\{h_{2m}^{[1]}\}_{m\in\N_0}\in \Gamma^1_{R_2(\theta_2)} (\N_0,
{\mathcal W}^1_{\widetilde{R}_1(\theta_1)} (S_1))$ for every direction $\theta_2$ on $S_2$.
So, taking into account (\ref{f2m}), (\ref{hatR}) and the first item in Remarks~\ref{obseW1RSE}, we deduce that
$\{f_{2m}-h_{2m}^{[1]}\}_{m\in\N_0}\in \Gamma^1_{R_2(\theta_2)} (\N_0,
{\mathcal W}^1_{\widehat{R}_1(\theta_1)} (S_1))$ for every direction $\theta_2$ on $S_2$.
Hence, the series $\varphi_2(z_2)=\sum_{m\in\N_0}(f_{2m}-h_{2m}^{[1]})z_2^m/m!$, with coefficients in ${\mathcal W}^1_{\widehat{R}_1(\theta_1)} (S_1)$, converges. Take $z_{02}$ as specified, and $S_{02}=S(\theta_{02}-\pi/2,\theta_{02}+\pi/2;\infty)\supset S_2$. We  consider the function $H_2^{[2]}:=\mathcal{L}_{z_{02}}^T(\varphi_2)$ which, again by Proposition~\ref{transLaplaW1Rtilde}, belongs to
${\mathcal W}^1_{\widetilde{R}_2(\theta_2)} (S_{02},{\mathcal W}^1_{\widehat{R}_1(\theta_1)} (S_1))$ and admits the series $\sum_{m\in\N_0}(f_{2m}-h_{2m}^{[1]})z_2^m$ as its asymptotic expansion in $S_{02}$. By Lemma~\ref{lemab}, the function $H^{[2]}:=\Psi_1(H_2^{[2]})$ belongs to
$${\mathcal W}^1_{(\widehat{R}_1(\theta_1),\widetilde{R}_2(\theta_2))} (S_{1}\times S_{02})\subset {\mathcal W}^1_{(\widehat{R}_1(\theta_1),\widetilde{R}_2(\theta_2))} (S_1\times S_2).$$
Write $\textrm{TA}_1(H^{[2]})=\{h_{1n}^{[2]},h_{2m}^{[2]}:n,m\in\N_0\}$.
Lemma~\ref{lemab} tells us that $h_{2m}^{[2]}=f_{2m}-h_{2m}^{[1]}$ for every $m\in\N_0$,
and Lemma 4.3 in~\cite{javi2} allows one to conclude that $h_{1n}^{[2]}\equiv 0$ for every $n\in\N_0$. So, the function $f:=H^{[1]}+H^{[2]}$ solves the problem, since
$f\in{\mathcal W}_{\widehat{\bR}(\btheta)}^1 (S)$ and, by linearity,
$\textrm{TA}_1(f)=\mathcal{F}_1$.
\end{proof}

We can now state and prove the main result of this paper. The
statement generalizes the one variable case (Theorem 1 in
\cite{Fruzha}).

\begin{teor}\label{teorfinal}
Let $S=S(\balpha, \bbeta; \brho )=\prod_{j=1}^n S_j$ be a polysector, $f\in \OO(S)$ asymptotically bounded in $S$, and ${\mathcal F}$ a coherent family in $S$ with first order subfamily ${\mathcal F}_1=\{f_{jm}:j\in\mathcal{N},\ m\in\N_0\}$.

\begin{enumerate}
\item[(i)] If $f$ has \sase\ following a multidirection $\btheta_0$ in $S$, given by the family ${\mathcal F}$, then $f$ has ${\mathcal F}$ as its family of \sase\ in $S$.
\item[(ii)] Suppose $\btheta_0-\frac{\pi}{2} \cdot
\boldsymbol{1} <\balpha<\btheta_0<\bbeta<\btheta_0+\frac{\pi}{2} \cdot
\boldsymbol{1}$, and let $f$ have ${\mathcal F}$ as family of
$\boldsymbol{1}$-Gevrey \sase\ following $\btheta_0=(\theta_{01},\ldots,\theta_{0n})$, with type $\bR_0(\btheta_0
)= (R_{01},\ldots ,R_{0n})$. For $j=1,\dots,n$, let $R_j:(\alpha_j,\beta_j)\to(0,\infty)$ be a function as in Definition~\ref{defW1RSE}, and suppose that for every $j=1,\ldots,n$ and every direction $\theta_j$ on $S_j$,
$$
\{ f_{jm}\}_{m\in \N_0}\in \Gamma^1_{R_j(\theta_j)} (\N_0,
{\mathcal W}_{\bR_{{j\,}'}(\btheta_{{j\,}'})}^1 (S_{{j\,}'})).
$$
Then, $f$ admits $\boldsymbol{1}$-Gevrey \sase\ in $S$, with type $\overline{\bR}(\btheta)$ in multidirection $\btheta$, where
$\overline{\bR}(\btheta) = (\overline{R}_1(\theta_{1}),\ldots ,\overline{R}_n(\theta_{n}))$
is defined as follows:

Let $\gamma:=g(\delta(0))=0.30028\ldots$, where $g$ is defined in~(\ref{tildeRg}). Put, for every $j=1,\dots,n$, $\gamma_j:=\sup_{\theta_j\in(\alpha_j,\beta_j)}R_j(\theta_j)$, $t_j(\theta_{0j}):=\min(R_{0j},\gamma_j\cdot \gamma,R_j(\theta_{0j}))$.
Then, we define
$$
\overline{R}_j(\theta_j)=\begin{cases}
    \min\big(t_j(\theta_{0j})\frac{\sin(\theta_j-\a_j)}{\sin(\theta_{0j}-\a_j)},
    R_j(\theta_j),\gamma_j\cos^2(\theta_j-\theta_{0j})g(\delta(\theta_j))\big) &\text{
if }\theta_j\in(\a_j,\theta_{0j}], \\\noalign{\vskip.5em}
%
	 \min\big(t_j(\theta_{0j})\frac{\sin(\theta_j-\b_j)}{\sin(\theta_{0j}-\b_j)},
    R_j(\theta_j),\gamma_j\cos^2(\theta_j-\theta_{0j})g(\delta(\theta_j))\big)
     &\text{
if }\theta_j\in(\theta_{0j},\b_j].
	\end{cases}
$$
\end{enumerate}
\end{teor}

\begin{proof}
(i) Borel-Ritt's theorem for \sases\ has been obtained in different ways, see~\cite{jesusisla,Majima2,javifelix}. It asserts the existence of a function $f_0\in\OO(S)$ admitting $\mathcal{F}$ as its family of \sase\ in $S$. So, the function $h:=f-f_0\in\OO(S)$ is asymptotically bounded in $S$ and admits null \sase\ following $\btheta_0$.
By Proposition~\ref{danulo}, $h$ has null \sase\ in $S$, and so $f=h+f_0$ admits \sase\ in $S$ given by the family $\mathcal{F}$.

(ii) Clearly, $\btheta_0$ is the bisecting direction of a polysector of opening $\pi\cdot\boldsymbol{1}$ and containing~$S$. Choose $\bzeta_0=(z_{01},\ldots,z_{0n})$ in the multidirection $\btheta_0$ such that $|z_{0j}|<\gamma_j$ for $j=1,\ldots,n$. From the hypotheses and from Lemma \ref{521}, there exists a function $F\in{\mathcal W}_{\widehat{\bR}(\btheta)}^1 (S)$
such that $\textrm{TA}_1(F)=\mathcal{F}_1$, where $\widehat{R}_j=\min(R_j,\widetilde{R}_j)$, $j=1,\ldots,n$, $\widetilde{\bR}$ being the function introduced in Proposition~\ref{transLaplaW1Rtilde}. In particular, by the second item in Remarks~\ref{obseW1RSE}, $F$ has $\boldsymbol{1}$-Gevrey \sase\ in $S$, of type $\widehat{\bR}(\btheta)$ in every multidirection $\btheta$. Then, the function $h:=f-F$ has null \sase\ following $\btheta_0$, of type
$$
\big(\min(R_{01},\widehat{R}_1(\theta_{01})),\ldots, \min(R_{0n},\widehat{R}_n(\theta_{0n}))\big).
$$
Note that for $j\in\mathcal{N}$,
$$
\widehat{R}_j(\theta_{0j})= \min(R_j(\theta_{0j}),\widetilde{R}_j(\theta_{0j}))=
\min(R_j(\theta_{0j}),|z_{0j}|\gamma).
$$
Using Proposition~\ref{58}, we deduce that $h$ has null \sase\ in $S$, and Remark~\ref{59} gives the type $\bR'(\btheta)$ of this expansion in every multidirection $\btheta=(\theta_1,\ldots,\theta_n)$ as:
$$
{R}'_j(\theta_j)=\begin{cases}
    \min(R_{0j},R_j(\theta_{0j}),|z_{0j}|\gamma) \frac{\sin(\theta_j-\a_j)}{\sin(\theta_{0j}-\a_j)} &\text{
if }\theta_j\in(\a_j,\theta_{0j}], \\\noalign{\vskip.5em}
\min(R_{0j},R_j(\theta_{0j}),|z_{0j}|\gamma) \frac{\sin(\theta_j-\b_j)}{\sin(\theta_{0j}-\b_j)}
     &\text{
if }\theta_j\in(\theta_{0j},\b_j].
	\end{cases}
$$
Finally, $f=F+h$ will admit $\boldsymbol{1}$-Gevrey \sase\ in $S$ of type
$\overline{\bR}(\btheta)$ in multidirection $\btheta$, where
$\overline{\bR}(\btheta) = (\overline{R}_1(\theta_{1}),\ldots ,\overline{R}_n(\theta_{n}))$ is obtained as
$$
\overline{R}_j(\theta_{j})=\min(\widehat{R}_j(\theta_j),{R}'_j(\theta_j)),\qquad j\in\mathcal{N}.
$$
In order to conclude, it suffices to make $|z_{0j}|$ tend to $\gamma_j$ for every $j$.

\end{proof}

\begin{obses}\linea
\begin{enumerate}
\item[(i)] One can treat the case of Gevrey \sase\ in wide polysectors, i.e., without the restriction $\btheta_0-\frac{\pi}{2} \cdot
\boldsymbol{1} <\balpha<\btheta_0<\bbeta<\btheta_0+\frac{\pi}{2} \cdot
\boldsymbol{1}$, with a technique similar to the one employed in~\cite{Fruzha}. However, the result becomes hard to state and we do not think it worthy to go into details.
\item[(ii)] All the results in this paper can be easily generalized for $\bk$-Gevrey \sases, for some $\bk=(k_1,\ldots,k_n)\in(0,\infty)^n$, as defined by Haraoka~\cite{Haraoka}.
\end{enumerate}
\end{obses}

\end{document}